\documentclass[10pt]{article}
\usepackage{verbatim}
\usepackage{tikz}
\usepackage{subfigure}
\usepackage[english]{babel}
\usepackage{graphicx}
\usepackage{extarrows}
\usepackage{lineno}
\usepackage{indentfirst,latexsym,bm,amssymb, amsmath,amsxtra,amscd}
\usepackage{amsthm}
\usepackage{enumitem}
\numberwithin{equation}{section}
\usepackage{makeidx}
\usepackage{geometry}
\usepackage{setspace}
\allowdisplaybreaks
\begin{document}

\title{Minimum Distance Spectral Radius of Tree with Given Maximum Degree}
\author{Yifei Li$^{*}$ ,  Zia Ullah Khan\thanks{College of Mathematics and Physics, Shanghai University of Electric Power, No.1851, Hucheng Ring Road, Pudong New Area, Shanghai, 201306, P.R. China. Email: 762928985@qq.com; zukhan@shiep.edu.cn}
\thanks{Corresponding Author: zukhan@shiep.edu.cn; zukhan@163.com}}
\date{}
\maketitle
\newtheorem{theorem0}{Theorem}
\newtheorem{theorem}{Lemma}
\newtheorem{theorem1}{Theorem}

\begin{abstract}
%Let $D(G)$ be the distance matrix of a simple graph $G$, and the largest eigenvalue of $D(G)$ be the distance spectral radius. The degree of vertex $v$ in $G$ is the number of edges incident to $v$. The order of $G$ is the number of vertices in $G$. 
In this article, the unique graph with minimum distance spectral radius among trees with given order $n$ and maximum degree $\varDelta$ is determined, where $2\varDelta+1 \leq n \leq \varDelta^{2} +1$ for $\varDelta \geq 3$. This article is a part proof of the conjecture proposed in $\cite{ref1}$.\vspace{1em}

\noindent \textit{Keyword:} distance spectral radius; Perron eigenvector; tree; maximum degree; complete $\varDelta$-ary tree\\
\noindent \textit{Subject classification:} 05C50
\end{abstract}

\section{Introduction}

\indent Let $G=(V,E)$ be a simple connected graph with $n = |V|$. For two vertices $v_{i}$ and $v_{j}$ ($i \ne j$) in $G$, the distance between $v_{i}$ and $v_{j}$, denoted by $d_{ij}$ or $dist(v_{i},v_{j})$, is the length of a shortest path joining $v_{i}$ and $v_{j}$. The diameter of a graph denoted by $diam (G)$ is the maximum distance between any two distinct vertices in $G$. For $v\in V(G)$, the degree of $v$ is the number of adjacent vertices of $v$ and is denoted by $deg(v)$. Let $N_{G}(v)$ be the set of vertices adjacent to $v$, i.e., the neighbours of $v$ in $G$, so $deg(v)= |N_{G}(v)|$.

\indent  The distance matrix of $G$ is the $n\times n$ matrix $D(G)$= ($d_{ij}$). The eigenvalues of $D(G)$ are known as the distance eigenvalues of $G$. Since $D(G)$ is real and symmetric, the distance eigenvalues of $G$ are all real. The distance spectral radius of $G$, denoted by $\rho(G)$, is the largest eigenvalue of $D(G)$, which is simple and positive. Since $D(G)$ is a non-negative irreducible matrix, by Perron-Frobenius Theorem $\cite{ref2}$, there is an eigenvector $x$ corresponding to $\rho(G)$ with positive coordinates, known as the Perron eigenvector of $D(G)$. For a vector $x$,
\begin{center}
$\rho(G) \geq \frac{x^{T}D(G)x}{x^{T}x} = \frac{2\sum\limits_{i=1}^{n}{\sum\limits_{j=i+1}^{n}{x_{v_{i}}x_{v_{j}}d_{ij}}}}{\sum\limits_{i=1}^{n}{x_{v_{i}}^{2}}}$ \\
\end{center}

\noindent with equality if and only if $x$ is the Perron eigenvector of $D(G)$. Throughout the paper, x will always be assumed as a unit vector, i.e., $\sum\limits_{i=1}^{n}{x_{v_{i}}^{2}}=1$.

\indent The study of distance eigenvalues was initiated by Graham et al.$\cite{ref3,ref4,ref5}$. Distance spectral radius can be used as molecular descriptor in QSPR modelling  as discussed by Consonni and Todeschini in $\cite{ref6}$, and to infer the extent of branching and model boiling points of alkanes in $\cite{ref7}$. In $\cite{ref8,ref9,ref10,ref11,ref12}$ authors have discussed in detail about the upper and lower bounds related to distance spectral radius. Subhi and Powers in $\cite{ref20}$ proved that for $n \geq 4$, the path $P_{n}$ has the maximum distance spectral radius among trees on $n$ vertices. Stevanovi\'{c} and Ili\'{c} in $\cite{ref1}$ proved that among trees with order n and given maximum degree $\varDelta$, the broom graph $B_{n,\varDelta}$ (consisting of a star $S_{\varDelta+1}$ and a path of length $n-\varDelta-1$ attached to an arbitrary pendent vertex of the star) is the extremal graph with maximum distance spectral radius. Furthermore, the authors proved that the star $S_{n}$ is the unique graph that minimizes distance spectral radius among tree on $n$ vertices. Bose et al. \cite{ref13} studied and analyzed the graphs which maximize and minimize $\rho(G)$ among the set of all trees with given order $n$ and fixed $r$ pendent vertices. Zhang and Godsil \cite{ref14} determine the graph which minimizes $\rho(G)$ for the set of all graphs with $k$ cut vertices. Yu et al. \cite{ref16} proposed a unique graph that minimizes $\rho(G)$ among the set of all unicyclic graphs on $n$ vertices; similarly Milan Nath and Somnath Paul \cite{ref17} proved a unique graph which minimizes $\rho(G)$ with a fixed matching number among the set of all connected bipartite graphs of order $n$. For detailed study of distance spectra the readers may refer to  \cite{ref13, ref19, ref16, ref14, ref18, ref15}. For more details on distance eigenvalues and especially regarding distance spectral radius, refer to the comprehensive survey of Aouchiche and Hansen $\cite{ref21}$.

\indent This paper is a part proof of the conjecture proposed by Stevanovi\'{c} and Ili\'{c} $\cite{ref1}$. In their paper, a complete $\varDelta$-ary tree is defined as: Start with the root having $\varDelta$ children. Every vertex different from the root, which is not in one of the last two levels, has exactly $\varDelta-1$ children. While in the last level all nodes need not exist, those that do fill the level consecutively (see Figure 1). Thus, at most one vertex on the level before the last has its degree different from $\varDelta$ and 1. They conjectured that a complete $\varDelta$-ary tree has the minimum distance spectral radius among trees on $n$ vertices with fixed maximum degree $\varDelta$. This paper studies the special case of the complete $\varDelta$-ary tree and will be helpful to settle the conjecture which has not been solved for more than a decade.

\begin{center}
\begin{tikzpicture}

\draw (6.5,3) -- (2.5,2) -- (1.5,1) -- (1,0);
\draw (1.5,1) -- (2,0);
\draw (2.5,2) -- (3.5,1) -- (3,0);
\draw (3.5,1) -- (4,0);

\draw (6.5,3) -- (6.5,2) -- (5.5,1) -- (5,0);
\draw (5.5,1) -- (6,0);
\draw (6.5,2) -- (7.5,1) -- (7,0);
\draw (7.5,1) -- (8,0);

\draw (6.5,3) -- (10.5,2) -- (9.5,1) -- (9,0);
\draw (10.5,2) -- (11.5,1);

\filldraw[fill=black,draw=black] (6.5,3) circle [radius=(0.0018)^(0.5)]; 
\filldraw[fill=black,draw=black] (2.5,2) circle [radius=(0.0018)^(0.5)];
\filldraw[fill=black,draw=black] (6.5,2) circle [radius=(0.0018)^(0.5)];
\filldraw[fill=black,draw=black] (10.5,2) circle [radius=(0.0018)^(0.5)];
\filldraw[fill=black,draw=black] (1.5,1) circle [radius=(0.0018)^(0.5)];
\filldraw[fill=black,draw=black] (3.5,1) circle [radius=(0.0018)^(0.5)]; 
\filldraw[fill=black,draw=black] (5.5,1) circle [radius=(0.0018)^(0.5)];
\filldraw[fill=black,draw=black] (7.5,1) circle [radius=(0.0018)^(0.5)]; 
\filldraw[fill=black,draw=black] (9.5,1) circle [radius=(0.0018)^(0.5)];
\filldraw[fill=black,draw=black] (11.5,1) circle [radius=(0.0018)^(0.5)]; 

\filldraw[fill=black,draw=black] (1,0) circle [radius=(0.0018)^(0.5)]; 
\filldraw[fill=black,draw=black] (2,0) circle [radius=(0.0018)^(0.5)]; 
\filldraw[fill=black,draw=black] (3,0) circle [radius=(0.0018)^(0.5)]; 
\filldraw[fill=black,draw=black] (4,0) circle [radius=(0.0018)^(0.5)]; 
\filldraw[fill=black,draw=black] (5,0) circle [radius=(0.0018)^(0.5)]; 
\filldraw[fill=black,draw=black] (6,0) circle [radius=(0.0018)^(0.5)]; 
\filldraw[fill=black,draw=black] (7,0) circle [radius=(0.0018)^(0.5)]; 
\filldraw[fill=black,draw=black] (8,0) circle [radius=(0.0018)^(0.5)]; 
\filldraw[fill=black,draw=black] (9,0) circle [radius=(0.0018)^(0.5)];

\node at (6,-1){$Figure\ 1: The\ complete\ 3-ary\ tree\ of\ order\ 19.$};
\end{tikzpicture}
\end{center}

\indent Wei $\cite{ref22}$ solved the conjecture under the condition of $\lceil \frac{n}{2} \rceil \leq \varDelta \leq n-1$, i.e., $\varDelta +1 \leq n \leq 2\varDelta$. This article further solves the conjecture under a less strict condition of relation between $n$ and $\varDelta$. \\

\indent The main result of the paper can be stated as follows:

\begin{theorem1}
Let $T_{(n,\varDelta)}$ be the tree with order n and maximum degree $\varDelta$, where $2\varDelta+1 \leq n \leq \varDelta^{2} +1$ and $\varDelta \geq 3$. $T\in T_{(n,\varDelta)}$ and $\rho(T)=\rho_{min}(T_{(n,\varDelta)})$. Then $T$ is the complete $\varDelta$-ary tree.
\end{theorem1}

\indent We list several new definitions which plays the base role in proving the theorems by means of transformation. Let $T$ be a tree and $P=v_{1}v_{2}\cdots v_{q}v_{q+1}$ be an arbitrary path of $T$, with $v_{1}$ and $v_{q+1}$ both being pendant vertices.

\begin{enumerate}
\item Let $\frac{1}{2}$-pendant vertex be a vertex which is adjacent to at least one pendant vertex. It's obvious that a $\frac{1}{2}$-pendant vertex can't be a pendant vertex at the same time in a tree with $n \geq 3$.

\item Let $ndp(T)$ be number of diametrical paths of $T$.

\item Let $v_{i(1)}=N_{T}(v_{i})\backslash \{v_{i-1},v_{i+1}\}$, $2\leq i \leq q$.

\item Assume $|v_{i(1)}|=m$, respectively denote them as $v_{i(1)(1)},\ \cdots,\ v_{i(1)(m)}$.

\item Let $G_{v_{i(1)(j)}}^{'}$ be the maximal subtree of $G-\{v_{i}\}$ that contains $v_{i(1)(j)}$.

\item Let $G_{v_{i}}=G_{v_{i(1)(1)}}^{'} \cup G_{v_{i(1)(2)}}^{'}\cup \cdots \cup G_{v_{i(1)(m)}}^{'} \cup v_{i}v_{i(1)(1)}\cup v_{i}v_{i(1)(2)}\cdots \cup v_{i}v_{i(1)(m)}$. $G_{v_{i}}^{p}$ is the subset of $G_{v_{i}}$ which is composed of all pendant vertices in $G_{v_{i}}$.

\item Let $v_{i(2)(j)}=N_{T}(v_{i(1)(j)})\backslash \{v_{i}\}$.

\item Assume $|v_{i(2)(j)}|=k$, respectively denote them as $v_{i(1)(j)(1)(1)},\ \cdots,\ v_{i(1)(j)(1)(k)}$. It's obvious $k = deg(v_{i(1)(j)}) - 1$.

\item Let $G_{v_{i(1)(j)}}=G_{v_{i(1)(j)}}^{'} - v_{i(1)(j)} \cup v_{i(1)(j)}v_{i(1)(j)(1)(1)} \cup v_{i(1)(j)}v_{i(1)(j)(1)(2)}\cdots  \cup v_{i(1)(j)}v_{i(1)(j)(1)(k)}$. $G_{v_{i(1)(j)}}^{p}$ is the subset of $G_{v_{i(1)(j)}}$ which is composed of all pendant vertices in $G_{v_{i(1)(j)}}$.

\item Let $v_{i(2)}=v_{i(2)(1)}\cup \cdots \cup v_{i(2)(m)}$.

\item Let $v_{i(3)(j)}=\{v|v\in G_{v_{i(1)(j)}}^{'}, dist(v,v_{i(1)(j)})=2\}$.

\item Let $v_{i(3)(j)(k)}=\{v|v\in v_{i(3)(j)}, dist(v,v_{i(1)(j)(1)(k)})=1\}$.

\item Assume $|v_{i(3)(j)(k)}|=k_{2}$, respectively denote them as $v_{i(1)(j)(1)(k)(1)(1)},\ \cdots,\ v_{i(1)(j)(1)(k)(1)(k_{2})}$. It's obvious $k_{2} = deg(v_{i(1)(j)(1)(k)}) - 1$.

\item Let $v_{i(3)}=v_{i(3)(1)}\cup \cdots \cup v_{i(3)(m)}$.

\item Pre-tree is a tree with $|v_{i(1)}| + |v_{i+1(1)}|\geq \varDelta-2$, where $v_{i}$ and $v_{i+1}$ are non-pendant neighbours.
\end{enumerate}

\indent Figure 2 clearly shows better understanding of the definitions via an example.
\begin{itemize}
\item $d_{1}=deg(v_{i})-2$, $d_{2}=deg(v_{i(1)(1)})-1$, $d_{3}=deg(v_{i(1)(d_{1})})-1$, $d_{4}=deg(v_{i(1)(1)(1)(1)})-1$.
\item $deg(v_{i(1)(j)})=1$ for $j\in[2,d_{1}-1]$, $deg(v_{i(1)(1)(1)(j)})=1$ for $j\in[2,d_{2}]$, $deg(v_{i(1)(d_{1})(1)(j)})=1$ for $j\in[1,d_{3}]$, $deg(v_{i(1)(1)(1)(1)(1)(j)})=1$ for $j\in[1,d_{4}]$.
\item $v_{i(1)}=v_{i(1)(1)}\cup v_{i(1)(2)}\cup \cdots \cup v_{i(1)(d_{1})}$, $v_{i(2)}=v_{i(2)(1)}\cup v_{i(2)(d_{1})}$, $v_{i(3)}=v_{i(3)(1)}=v_{i(3)(1)(1)}$.
\item Among vertices of $v_{i(1)}$, there're two $\frac{1}{2}$-pendant vertices which are $v_{i(1)(1)}$ and $v_{i(1)(d_{1})}$.
\item Among vertices of $v_{i(2)}$, there's one $\frac{1}{2}$-pendant vertex which is $v_{i(1)(1)(1)(1)}$.
\end{itemize}

\begin{center}
\begin{tikzpicture}
\draw (0,3)node[above]{$v_{i(1)(1)(1)(1)(1)(1)}$} -- (2,2);
\filldraw[fill=black,draw=black] (0,3) circle [radius=(0.0018)^(0.5)]; 
\filldraw[fill=black,draw=black](1.9,3)circle [radius=0.01];
\filldraw[fill=black,draw=black](2,3)circle [radius=0.01];
\filldraw[fill=black,draw=black](2.1,3)circle [radius=0.01]; 
\filldraw[fill=black,draw=black] (4,3) circle [radius=(0.0018)^(0.5)]; 
\draw (4,3)node[above]{$v_{i(1)(1)(1)(1)(1)(d_{4})}$} -- (2,2);

\draw (2,2)node[left]{$v_{i(1)(1)(1)(1)}$} -- (4,1);
\filldraw[fill=black,draw=black] (2,2) circle [radius=(0.0018)^(0.5)]; 
\filldraw[fill=black,draw=black](3.9,2)circle [radius=0.01];
\filldraw[fill=black,draw=black](4,2)circle [radius=0.01];
\filldraw[fill=black,draw=black](4.1,2)circle [radius=0.01]; 
\filldraw[fill=black,draw=black] (6,2) circle [radius=(0.0018)^(0.5)]; 
\draw (6,2)node[above]{$v_{i(1)(1)(1)(d_{2})}$} -- (4,1);

\draw (4,1)node[left]{$v_{i(1)(1)}$} -- (7,0);
\filldraw[fill=black,draw=black] (4,1) circle [radius=(0.0018)^(0.5)]; 
\filldraw[fill=black,draw=black](6.9,0.5)circle [radius=0.01];
\filldraw[fill=black,draw=black](7,0.5)circle [radius=0.01];
\filldraw[fill=black,draw=black](7.1,0.5)circle [radius=0.01]; 
\filldraw[fill=black,draw=black] (7,0) circle [radius=(0.0018)^(0.5)]; 
\draw (10,1) -- (7,0);

\draw (8,2)node[above]{$v_{i(1)(d_{1})(1)(1)}$} -- (10,1)node[right]{$v_{i(1)(d_{1})}$};
\filldraw[fill=black,draw=black] (8,2) circle [radius=(0.0018)^(0.5)]; 
\filldraw[fill=black,draw=black](9.9,2)circle [radius=0.01];
\filldraw[fill=black,draw=black](10,2)circle [radius=0.01];
\filldraw[fill=black,draw=black](10.1,2)circle [radius=0.01]; 
\filldraw[fill=black,draw=black] (12,2) circle [radius=(0.0018)^(0.5)]; 
\filldraw[fill=black,draw=black] (10,1) circle [radius=(0.0018)^(0.5)]; 
\draw (12,2)node[above]{$v_{i(1)(d_{1})(1)(d_{3})}$} -- (10,1);

\draw (1,0)node[below]{$v_{1}$} -- (3,0)node[below]{$v_{2}$};
\filldraw[fill=black,draw=black] (1,0) circle [radius=(0.0018)^(0.5)]; 
\filldraw[fill=black,draw=black] (3,0) circle [radius=(0.0018)^(0.5)]; 
\filldraw[fill=black,draw=black](3.9,0)circle [radius=0.01];
\filldraw[fill=black,draw=black](4,0)circle [radius=0.01];
\filldraw[fill=black,draw=black](4.1,0)circle [radius=0.01]; 
\draw (5,0)node[below]{$v_{i-1}$} -- (7,0)node[below]{$v_{i}$} -- (9,0)node[below]{$v_{i+1}$};
\filldraw[fill=black,draw=black] (5,0) circle [radius=(0.0018)^(0.5)];
\filldraw[fill=black,draw=black] (7,0) circle [radius=(0.0018)^(0.5)];  
\filldraw[fill=black,draw=black] (9,0) circle [radius=(0.0018)^(0.5)]; 
\filldraw[fill=black,draw=black](9.9,0)circle [radius=0.01];
\filldraw[fill=black,draw=black](10,0)circle [radius=0.01];
\filldraw[fill=black,draw=black](10.1,0)circle [radius=0.01]; 
\draw (11,0)node[below]{$v_{q}$} -- (13,0)node[below]{$v_{q+1}$};
\filldraw[fill=black,draw=black] (11,0) circle [radius=(0.0018)^(0.5)]; 
\filldraw[fill=black,draw=black] (13,0) circle [radius=(0.0018)^(0.5)]; 

\draw[dashed] (-1.3,3.5) rectangle (6.9,0.9);
\node at (-2,2){$G_{v_{i(1)(1)}}^{'}$};
\draw[dashed] (7,2.5) rectangle (13,0.9);
\node at (14,1.5){$G_{v_{i(1)(d_{1})}}^{'}$};
\node at (6,-1){$T$};
\node at (6,-1.5){$Figure\ 2: The\ graph\ T.$};
\end{tikzpicture}
\end{center}

\section{Preliminaries}

\begin{theorem}\label{lem1}
$\cite{ref1}$ Denote by $G(v,k)$ the graph obtained from $G$ by attaching at the vertex $v$ a pendent path $P = vv_{1}v_{2} \cdots v_{k}$ of length $k$. Let $x_{v_{0}}$, $x_{v_{1}}$, $x_{v_{2}},\cdots, x_{v_{k}}$ be the coordinates of the Perron eigenvector for vertices $v, v_{1}, v_{2},\cdots, v_{k}$ respectively, and let $S$ be the sum of all coordinates. Then,
$$\sum\limits_{i=1}^k{x_{v_{i}}}=x_{v_{0}}\cdot f(k)+\frac{S}{\rho}\cdot g(k)$$\\
where $f(x)=\frac{t(t^{2x}-1)}{(1+t^{2x+1})(t-1)}$ and $g(x)=\frac{t(t^{x}-1)(t^{x+1}-1)}{(1+t^{2x+1})(t-1)^{2}}$,
with only solution for corresponding recurrent equation being $t=1+\frac{1}{\rho}+\frac{\sqrt{2\rho+1}}{\rho}$. Furthermore, the functions $f(x)$ and $g(x)$
are monotonically increasing.
\end{theorem}

\begin{theorem}\label{lem2}
$\cite{ref23}$ Let $G$ be a graph that contains a cut edge $e = v_{i}v_{i+1}$. Let $x$ be the Perron eigenvector of $G$. We assume that the coordinates of $x$ are indexed by the vertices of $G$. Then
$$\rho(G)\cdot(x_{v_{i}}-x_{v_{i+1}})=S_{i+1}-S_{i},$$
where $S_{i+1}$($S_{i}$, respectively) is the sum of all coordinates of $x$ corresponding to the vertices in the component that contains vertex $v_{i+1}$ ($v_{i}$, respectively), after removing the bridge $v_{i}v_{i+1}$ in $G$.
\end{theorem}

\begin{theorem}\label{lem3}
\cite{ref20} Let $G$ be a graph and $x$ be the Perron eigenvector of $G$. Assume $u$,$v$,$w \in V(G)$ and $uv$,$vw \in E(G)$. If $x_{u} < x_{v}$ and one of the edges $uv$ and $vw$ is a cut edge, then $x_{v} < x_{w}$.
\end{theorem}

\begin{theorem}\label{lem4}
$\cite{ref20}$ Let $G$ be a graph and $x$ be the Perron eigenvector of $G$. Assume $v$,$w \in V(G)$ and $vw \in E(G)$. If $w$ is a pendant vertex and $v$ is a non-pendant vertex, then $x_{w} > x_{v}$.
\end{theorem}

\begin{theorem}\label{lem5}
$\cite{ref20}$ Let $T$ be a tree with order $n\geq3$ and $x$ be the Perron eigenvector of $T$. Assume $v_{1}v_{2}\cdots v_{q}$ be an arbitrary path in $T$, with $v_{q}$ be a pendant vertex. If $x_{v_{1}} < x_{v_{2}}$, then $x_{v_{2}} < x_{v_{3}} < \cdots < x_{v_{q}}$.
\end{theorem}

\begin{theorem}\label{lem6}
$\cite{ref24}$ Let $G$ and $G^{*}$ be graphs as shown in Figure 3. If $deg_{G}(u)$ and $deg_{G}(v)\geq2$, then $\rho(G)>\rho(G^{*})$.
\end{theorem}
\begin{center}
\begin{tikzpicture}
\draw (1,0) ellipse[x radius=1,y radius=0.3];

\draw (2,0)node[above]{$u$} -- (2.5,0)node[above]{$v$};

\filldraw[fill=black,draw=black] (2,0) circle [radius=(0.0018)^(0.5)]; 

\filldraw[fill=black,draw=black] (2.5,0) circle [radius=(0.0018)^(0.5)]; 

\draw (3.5,0) ellipse[x radius=1,y radius=0.3];

\node at (2.25,-0.5){$G$};

\draw (8,0) ellipse[x radius=1,y radius=0.3];

\draw (9,0)node[left]{$u$} -- (9,0.5)node[above]{$v$};

\filldraw[fill=black,draw=black] (9,0) circle [radius=(0.0018)^(0.5)]; 

\filldraw[fill=black,draw=black] (9,0.5) circle [radius=(0.0018)^(0.5)]; 

\draw (10,0) ellipse[x radius=1,y radius=0.3];

\node at (9,-0.5){$G^{*}$};

\node at (5.625,-1.5){$Figure\ 3: The\ graphs\ G\ and\ G^{*}\ in\ Lemma\ 6.$};
\end{tikzpicture}
\end{center}

\begin{theorem}\label{lem7}
\cite{ref25} Let $G_{1}$ and $G_{2}$ be trees as shown in Figure 4. If $deg_{G_{1}}(u)$ and $deg_{G_{1}}(v)\geq2$, then $\rho(G_{1})>\rho(G_{2})$.
\end{theorem}

\begin{center}
\begin{tikzpicture}

\draw (0,0.5)node[left]{$w_{1}$} -- (0.5,0)node[above]{$u$} -- (1.5,0) -- (2.5,0)node[above]{$v$}; 
\draw (0,-0.5)node[left]{$w_{t}$} -- (0.5,0);

\filldraw[fill=black,draw=black] (-0.03,0.53) circle [radius=(0.0018)^(0.5)]; 
\filldraw[fill=black,draw=black] (-0.03,-0.53) circle [radius=(0.0018)^(0.5)]; 

\filldraw[fill=black,draw=black](-0.03,0.25)circle [radius=0.01];
\filldraw[fill=black,draw=black](-0.03,0)circle [radius=0.01];
\filldraw[fill=black,draw=black](-0.03,-0.25)circle [radius=0.01];

\filldraw[fill=black,draw=black] (0.5,0) circle [radius=(0.0018)^(0.5)]; 
\filldraw[fill=black,draw=black] (1.5,0) circle [radius=(0.0018)^(0.5)]; 
\filldraw[fill=black,draw=black] (2.5,0) circle [radius=(0.0018)^(0.5)]; 

\draw (1.5,1) ellipse[x radius=0.3,y radius=1];
\draw (3.5,0) ellipse[x radius=1,y radius=0.3];
\node at (1.5,1){$P$};
\node at (3.5,0){$Q$};

\node at (1.5,-0.5){$G_{1}$};

\draw (6.5,0)node[above]{$u$} -- (7.5,0) -- (8.5,0)node[above]{$v$} -- (8,0.7859)node[above]{$w_{1}$}; 
\draw (8.5,0) -- (9,0.7859)node[above]{$w_{t}$};

\filldraw[fill=black,draw=black] (7.97,0.8159) circle [radius=(0.0018)^(0.5)]; 
\filldraw[fill=black,draw=black] (9.03,0.8159) circle [radius=(0.0018)^(0.5)]; 

\filldraw[fill=black,draw=black](8.235,0.8159)circle [radius=0.01];
\filldraw[fill=black,draw=black](8.5,0.8159)circle [radius=0.01];
\filldraw[fill=black,draw=black](8.765,0.8159)circle [radius=0.01];

\filldraw[fill=black,draw=black] (6.5,0) circle [radius=(0.0018)^(0.5)]; 
\filldraw[fill=black,draw=black] (7.5,0) circle [radius=(0.0018)^(0.5)]; 
\filldraw[fill=black,draw=black] (8.5,0) circle [radius=(0.0018)^(0.5)]; 

\draw (7.5,1) ellipse[x radius=0.3,y radius=1];
\draw (9.5,0) ellipse[x radius=1,y radius=0.3];
\node at (7.5,1){$P$};
\node at (9.5,0){$Q$};

\node at (8.5,-0.5){$G_{2}$};

\node at (5.25,-1.5){$Figure\ 4: The\ graphs\ G_{1}\ and\ G_{2}\ in\ Lemma\ 7.$};
\end{tikzpicture}
\end{center}

\begin{theorem}\label{lem8}
\cite{ref22} Let $T_{(n,\varDelta)}$ be the tree with order $n$ and maximum degree $\varDelta$. Let $T\in T_{(n,\varDelta)}$ and $\rho(T)=\rho_{min}(T_{(n,\varDelta)})$. Let $P=v_{1}v_{2}v_{3}\ldots v_{d-1}v_{d}v_{d+1}$ be a diametrical path of $T$. Then $|v_{i(1)}|+|v_{i+1(1)}|\geq \varDelta-2\ (2\leq i \leq d-1))$.
\end{theorem}

\begin{center}
\begin{tikzpicture}
\draw (0,0)node[below]{$v_{1}$} -- (2,0)node[below]{$v_{2}$};
\draw (4,0)node[below]{$v_{i}$} -- (6,0)node[below]{$v_{i+1}$} -- (8,0)node[below]{$v_{i+2}$};
\draw (10,0)node[below]{$v_{d}$} -- (12,0)node[below]{$v_{d+1}$};

\filldraw[fill=black,draw=black](2.5,0)circle [radius=0.01];
\filldraw[fill=black,draw=black](3,0)circle [radius=0.01];
\filldraw[fill=black,draw=black](3.5,0)circle [radius=0.01];

\filldraw[fill=black,draw=black](8.5,0)circle [radius=0.01];
\filldraw[fill=black,draw=black](9,0)circle [radius=0.01];
\filldraw[fill=black,draw=black](9.5,0)circle [radius=0.01];

\filldraw[fill=black,draw=black] (0,0) circle [radius=(0.0018)^(0.5)]; 
\filldraw[fill=black,draw=black] (2,0) circle [radius=(0.0018)^(0.5)]; 
\filldraw[fill=black,draw=black] (4,0) circle [radius=(0.0018)^(0.5)]; 
\filldraw[fill=black,draw=black] (6,0) circle [radius=(0.0018)^(0.5)]; 
\filldraw[fill=black,draw=black] (8,0) circle [radius=(0.0018)^(0.5)]; 
\filldraw[fill=black,draw=black] (10,0) circle [radius=(0.0018)^(0.5)]; 
\filldraw[fill=black,draw=black] (12,0) circle [radius=(0.0018)^(0.5)]; 

\draw (2,0.5) circle[radius=0.5];
\draw (4,0.5) circle[radius=0.5];
\draw (6,0.5) circle[radius=0.5];
\draw (8,0.5) circle[radius=0.5];
\draw (10,0.5) circle[radius=0.5];
\node at (2,0.5){$G_{v_{2}}$};
\node at (4,0.5){$G_{v_{i}}$};
\node at (6,0.5){$G_{v_{i+1}}$};
\node at (8,0.5){$G_{v_{i+2}}$};
\node at (10,0.5){$G_{v_{d}}$};
%T
\node at (6,-0.6){$T$};

\node at (6,-1.6){$Figure\ 5: The\ graph\ T\ in\ Lemma\ 8.$};
\end{tikzpicture}
\end{center}

\begin{theorem}\label{lem9}
Let $G$ be graph as shown in Figure 6, which satisfies following conditions:
\begin{enumerate}
\item  $H$ is any graph.
\item  There're $p$ pendant vertices adjacent to $u$.
\item  There're $q$ pendant vertices adjacent to $w$.
\item  There's 1 pendant vertex $w_{2}$ adjacent to $v$, and $w_{2} \notin H$.
\item  $p>q\geq1$.
\end{enumerate}
Let $x$ be Perron eigenvector of $G$. Then $x_{u}< x_{w} < x_{w_{2}}$, $x_{u} + \sum\limits_{i=1}^p{x_{a_{i}}}>x_{w} + \sum\limits_{i=1}^q{x_{b_{i}}}$ and $\sum\limits_{i=1}^p{x_{a_{i}}}>\sum\limits_{i=1}^q{x_{b_{i}}}$. 
\end{theorem}

\begin{center}
\begin{tikzpicture}
\draw (0,1)node[left]{$a_{1}$} -- (1,0)node[above]{$u$} -- (3,0)node[above]{$v$} -- (5,0)node[above]{$w$} -- (6,1)node[right]{$b_{1}$};
\draw (0,-1)node[left]{$a_{p}$} -- (1,0);
\draw (6,-1)node[left]{$b_{q}$} -- (5,0);

\filldraw[fill=black,draw=black](0.5,0.3)circle [radius=0.01];
\filldraw[fill=black,draw=black](0.5,0)circle [radius=0.01];
\filldraw[fill=black,draw=black](0.5,-0.3)circle [radius=0.01];

\filldraw[fill=black,draw=black](5.5,0.3)circle [radius=0.01];
\filldraw[fill=black,draw=black](5.5,0)circle [radius=0.01];
\filldraw[fill=black,draw=black](5.5,-0.3)circle [radius=0.01];

\filldraw[fill=black,draw=black] (0,1) circle [radius=(0.0018)^(0.5)]; 
\filldraw[fill=black,draw=black] (1,0) circle [radius=(0.0018)^(0.5)]; 
\filldraw[fill=black,draw=black] (3,0) circle [radius=(0.0018)^(0.5)]; 
\filldraw[fill=black,draw=black] (5,0) circle [radius=(0.0018)^(0.5)]; 
\filldraw[fill=black,draw=black] (6,1) circle [radius=(0.0018)^(0.5)]; 
\filldraw[fill=black,draw=black] (0,-1) circle [radius=(0.0018)^(0.5)]; 
\filldraw[fill=black,draw=black] (6,-1) circle [radius=(0.0018)^(0.5)]; 

\draw (3,-0.5) ellipse[x radius=0.3,y radius=0.5];
\node at (3,-0.5){$H$};

\draw (4,1)node[left]{$w_{2}$} -- (3,0);
\filldraw[fill=black,draw=black] (4,1) circle [radius=(0.0018)^(0.5)];

\node at (3,-2){$G$};

\node at (3,-3){$Figure\ 6: The\ graph\ G\ in\ Lemma\ 9.$};
\end{tikzpicture}
\end{center}

\begin{proof}
Suppose on the contrary that $x_{u}\geq x_{w}$, then since $p>q$, by Lemma \ref{lem1} we can get following result:

\begin{equation}\label{9.1}
\sum\limits_{i=1}^p{x_{a_{i}}}>\sum\limits_{i=1}^q{x_{b_{i}}}.
\end{equation}

Thus from inequality (\ref{9.1}), we get
\begin{equation}\label{9.2}
0\leq \rho(G)\cdot(x_{u}-x_{w})=2[x_{w}+\sum\limits_{i=1}^q{x_{b_{i}}}-x_{u}-\sum\limits_{i=1}^p{x_{a_{i}}}]<0,
\end{equation}
which is a contradiction. So we have
\begin{equation}\label{9.3}
x_{u}< x_{w}.
\end{equation}
\indent By (\ref{9.3}), we have
\begin{equation}\label{9.3.1}
x_{u} + \sum\limits_{i=1}^p{x_{a_{i}}}>x_{w} + \sum\limits_{i=1}^q{x_{b_{i}}}.
\end{equation}
\indent Combining (\ref{9.3}) and (\ref{9.3.1}), we can easily get the conclusion that
\begin{equation}\label{9.4}
\sum\limits_{i=1}^p{x_{a_{i}}}>\sum\limits_{i=1}^q{x_{b_{i}}}.
\end{equation}
\indent Similar to above, we can easily show that $x_{u}< x_{w} < x_{w_{2}}$. It's clear that inequalities (\ref{9.3}), (\ref{9.3.1}) and (\ref{9.4}) still hold in $G^{*}=G- \{w_{2}\}$.
\end{proof}

\begin{theorem}\label{lem10}
Let $G$ be graph as shown in Figure 7, which satisfies following conditions: 
\begin{enumerate}
\item  $H$ is any graph.
\item  There're $p$ pendant vertices adjacent to $u$.
\item  There're $q$ pendant vertices adjacent to $w$.
\item  $p\geq q\geq 1$.
\end{enumerate}
Let $G^{*}=G - wb_{q}+ ub_{q}$. Then $\rho(G)>\rho(G^{*})$.
\end{theorem}

\begin{center}
\begin{tikzpicture}
\draw (0,0.5)node[left]{$a_{1}$} -- (0.5,0)node[above]{$u$} -- (1.5,0)node[above]{$v$} -- (2.5,0)node[above]{$w$} -- (3,0.5)node[right]{$b_{1}$};
\draw (0,-0.5)node[left]{$a_{p}$} -- (0.5,0);
\draw (2.5,0) -- (3,-0.5)node[right]{$b_{q}$};

\draw[thick] (-0.01,0.51) circle [radius=(0.001)^(0.5)]; 

\filldraw[fill=black,draw=black](0.25,0.1)circle [radius=0.01];
\filldraw[fill=black,draw=black](0.25,0)circle [radius=0.01];
\filldraw[fill=black,draw=black](0.25,-0.1)circle [radius=0.01];

\draw[thick] (-0.01,-0.51) circle [radius=(0.001)^(0.5)];

\draw[thick] (3.01,0.51) circle [radius=(0.001)^(0.5)];
 
\filldraw[fill=black,draw=black](2.75,0.1)circle [radius=0.01];
\filldraw[fill=black,draw=black](2.75,0)circle [radius=0.01];
\filldraw[fill=black,draw=black](2.75,-0.1)circle [radius=0.01];

\draw[thick] (3.01,-0.51) circle [radius=(0.001)^(0.5)];

\draw[thick] (0.51,-0.02) circle [radius=(0.001)^(0.5)]; 
\draw[thick] (1.5,-0.02) circle [radius=(0.001)^(0.5)]; 
\draw (1.5,-0.56) circle[radius=0.5];
\node at (1.5,-0.56){$H$};
\draw[thick] (2.49,-0.02) circle [radius=(0.001)^(0.5)];

\node at (1.5,-1.5){$G$};

\draw (7,0.5)node[left]{$a_{1}$} -- (7.5,0)node[above]{$u$} -- (8.5,0)node[above]{$v$} -- (9.5,0)node[above]{$w$} -- (10,0.5)node[right]{$b_{1}$};
\draw (7,-0.5)node[left]{$a_{p}$} -- (7.5,0);
\draw (7.2,-0.5)node[below]{$b_{q}$} -- (7.5,0);
\draw (9.5,0) -- (10,-0.5)node[right]{$b_{q-1}$};

\draw[thick] (6.99,0.51) circle [radius=(0.001)^(0.5)]; 

\filldraw[fill=black,draw=black](7.25,0.1)circle [radius=0.01];
\filldraw[fill=black,draw=black](7.25,0)circle [radius=0.01];
\filldraw[fill=black,draw=black](7.25,-0.1)circle [radius=0.01];

\draw[thick] (6.99,-0.51) circle [radius=(0.001)^(0.5)]; 
\draw[thick] (7.19,-0.51) circle [radius=(0.001)^(0.5)]; 

\draw[thick] (10.01,0.51) circle [radius=(0.001)^(0.5)];
 
\filldraw[fill=black,draw=black](9.75,0.1)circle [radius=0.01];
\filldraw[fill=black,draw=black](9.75,0)circle [radius=0.01];
\filldraw[fill=black,draw=black](9.75,-0.1)circle [radius=0.01];

\draw[thick] (10.01,-0.51) circle [radius=(0.001)^(0.5)];

\draw[thick] (7.51,-0.02) circle [radius=(0.001)^(0.5)]; 
\draw[thick] (8.5,-0.02) circle [radius=(0.001)^(0.5)]; 
\draw (8.5,-0.56) circle[radius=0.5];
\node at (8.5,-0.56){$H$};
\draw[thick] (9.49,-0.02) circle [radius=(0.001)^(0.5)];

\node at (8.5,-1.5){$G^{*}$};

\node at (5,-2.5){$Figure\ 7: The\ graphs\ G\ and\ G^{*}\ in\ Lemma\ 10.$};
\end{tikzpicture}
\end{center}
\begin{proof}
Let $x$ be Perron eigenvector of $G^{*}$. We can see from Lemma \ref{lem9}, $x_{u}<x_{w}$ and $x_{b_{q}}<x_{b_{1}}$, using which we can determine the sign of $x_{u}+\sum\limits_{i=1}^p{x_{a_{i}}}- x_{w}-\sum\limits_{i=1}^{q-1}{x_{b_{i}}}$.

\begin{eqnarray}
& &\rho(G^{*})[x_{u}+\sum\limits_{i=1}^p{x_{a_{i}}}-x_{w}-\sum\limits_{i=1}^{q-1}{x_{b_{i}}}] \nonumber \\
&>&[0+p-2-3(q-1)] x_{u}+[1+0+2(p-1)-3-4(q-1)]p{x_{a_{1}}} +[1+2p-3-4(q-1)]{x_{b_{q}}}\nonumber \\
& &+[2+3p-0-(q-1)]x_{w} +[3+4p-1-0-2(q-2)](q-1){x_{b_{1}}}\nonumber \\
&=&(p-3q+1)x_{u}+(2p-4q)p{x_{a_{1}}}+ (2p-4q+2){x_{b_{q}}}+(3p-q+3)x_{w}+(4p-2q+6)(q - 1){x_{b_{1}}}\nonumber \\
\Rightarrow& &[\rho(G^{*})+2q][x_{u}+\sum\limits_{i=1}^p{x_{a_{i}}}-x_{w}-\sum\limits_{i=1}^{q-1}{x_{b_{i}}}]\nonumber \\
&>&(p-q+1)x_{u}+2(p-q)p{x_{a_{1}}}+ 2(p-2q+1){x_{b_{q}}}+3(p-q+1)x_{w} +2(2p-2q+3)(q - 1){x_{b_{1}}} \nonumber \\
&>&(p-q+1)x_{u}+2(p-q)p{x_{a_{1}}}+ 2(p+q-2){x_{b_{q}}}+3(p-q+1)x_{w} +4(p-q)(q - 1){x_{b_{1}}} \nonumber \\
&>&0\nonumber \\
\Rightarrow& &x_{u}+\sum\limits_{i=1}^p{x_{a_{i}}}-x_{w}-\sum\limits_{i=1}^{q-1}{x_{b_{i}}}>0.\label{10.1}
\end{eqnarray}
\indent By (\ref{10.1}), we can determine the sign of $\rho(G)-\rho(G^{*})$.
\begin{eqnarray*}
\frac{1}{2}[\rho(G)-\rho(G^{*})]\geq \frac{1}{2}x^{t}[D(G)-D(G^{*})]x =2{x_{b_{q}}}[x_{u}+\sum\limits_{i=1}^p{x_{a_{i}}}-x_{w}-\sum\limits_{i=1}^{q-1}{x_{b_{i}}}] >0\Rightarrow \rho(G)>\rho(G^{*}).
\end{eqnarray*}
\end{proof}

\begin{theorem}\label{lem11}
Let $G$ be graph as shown in Figure 8, which satisfies following conditions: 
\begin{enumerate}
\item  There're $p$ pendant vertices adjacent to $u$.
\item  There're $q$ pendant vertices adjacent to $w$.
\item  There're $r$ pendant vertices adjacent to $v_{1}$.
\item  There're $r$ pendant vertices adjacent to $v_{2}$.
\item  $p\geq q\geq 1$.
\item  $r\geq p+q-1$.
\end{enumerate}
Let $G^{*}=G-\sum\limits_{i=1}^{q}wb_{i}+ \sum\limits_{i=1}^{q}ub_{i}$. Then $\rho(G)>\rho(G^{*})$.
\end{theorem}

\begin{center}
\begin{tikzpicture}
\draw (0,0.5)node[left]{$a_{1}$} -- (0.5,0)node[below]{$u$} -- (1.5,0)node[below]{$v_{1}$} -- (2.5,0)node[below]{$v_{2}$} -- (3.5,0)node[below]{$w$}-- (4,0.5)node[right]{$b_{1}$};

\draw (0,-0.5)node[left]{$a_{p}$} -- (0.5,0);

\draw (3.5,0) -- (4,-0.5)node[right]{$b_{q}$};

\draw (1.5,0) -- (1.25,0.5)node[above]{$c_{1}$};

\draw (1.5,0) -- (1.75,0.5)node[above]{$c_{r}$};

\draw[thick] (1.249,0.51) circle [radius=(0.001)^(0.5)]; 
\draw[thick] (1.751,0.51) circle [radius=(0.001)^(0.5)]; 

\filldraw[fill=black,draw=black](1.4,0.25)circle [radius=0.01];
\filldraw[fill=black,draw=black](1.5,0.25)circle [radius=0.01];
\filldraw[fill=black,draw=black](1.6,0.25)circle [radius=0.01];

\draw (2.5,0) -- (2.25,0.5)node[above]{$d_{1}$};

\draw (2.5,0) -- (2.75,0.5)node[above]{$d_{r}$};

\draw[thick] (2.249,0.51) circle [radius=(0.001)^(0.5)]; 
\draw[thick] (2.751,0.51) circle [radius=(0.001)^(0.5)]; 

\filldraw[fill=black,draw=black](2.4,0.25)circle [radius=0.01];
\filldraw[fill=black,draw=black](2.5,0.25)circle [radius=0.01];
\filldraw[fill=black,draw=black](2.6,0.25)circle [radius=0.01];

\draw[thick] (-0.01,0.51) circle [radius=(0.001)^(0.5)]; 
\draw[thick] (-0.01,-0.51) circle [radius=(0.001)^(0.5)]; 

\filldraw[fill=black,draw=black](0.25,0.1)circle [radius=0.01];
\filldraw[fill=black,draw=black](0.25,0)circle [radius=0.01];
\filldraw[fill=black,draw=black](0.25,-0.1)circle [radius=0.01];

\draw[thick] (4.01,0.51) circle [radius=(0.001)^(0.5)];
\draw[thick] (4.01,-0.51) circle [radius=(0.001)^(0.5)]; 

\filldraw[fill=black,draw=black](3.75,0.1)circle [radius=0.01];
\filldraw[fill=black,draw=black](3.75,0)circle [radius=0.01];
\filldraw[fill=black,draw=black](3.75,-0.1)circle [radius=0.01];

\draw[thick] (0.51,-0.02) circle [radius=(0.001)^(0.5)]; 
\draw[thick] (1.5,-0.02) circle [radius=(0.001)^(0.5)]; 
\draw[thick] (2.5,-0.02) circle [radius=(0.001)^(0.5)]; 
\draw[thick] (3.49,-0.02) circle [radius=(0.001)^(0.5)];

\node at (2,-0.5){$G$};

\draw (7,0.5)node[left]{$a_{1}$} -- (7.5,0)node[below]{$u$} -- (8.5,0)node[below]{$v_{1}$} -- (9.5,0)node[below]{$v_{2}$} -- (10.5,0)node[below]{$w$};

\draw (7,0.1)node[left]{$a_{p}$} -- (7.5,0);

\draw (7,-0.1)node[left]{$b_{1}$} -- (7.5,0);

\draw (7,-0.5)node[left]{$b_{q}$} -- (7.5,0);

\draw (8.5,0) -- (8.25,0.5)node[above]{$c_{1}$};

\draw (8.5,0) -- (8.75,0.5)node[above]{$c_{r}$};

\draw[thick] (8.249,0.51) circle [radius=(0.001)^(0.5)]; 
\draw[thick] (8.751,0.51) circle [radius=(0.001)^(0.5)]; 

\filldraw[fill=black,draw=black](8.4,0.25)circle [radius=0.01];
\filldraw[fill=black,draw=black](8.5,0.25)circle [radius=0.01];
\filldraw[fill=black,draw=black](8.6,0.25)circle [radius=0.01];

\draw (9.5,0) -- (9.25,0.5)node[above]{$d_{1}$};

\draw (9.5,0) -- (9.75,0.5)node[above]{$d_{r}$};

\draw[thick] (9.249,0.51) circle [radius=(0.001)^(0.5)]; 
\draw[thick] (9.751,0.51) circle [radius=(0.001)^(0.5)]; 

\filldraw[fill=black,draw=black](9.4,0.25)circle [radius=0.01];
\filldraw[fill=black,draw=black](9.5,0.25)circle [radius=0.01];
\filldraw[fill=black,draw=black](9.6,0.25)circle [radius=0.01];

\draw[thick] (6.99,0.51) circle [radius=(0.001)^(0.5)]; 
\draw[thick] (6.99,0.11) circle [radius=(0.001)^(0.5)]; 

\filldraw[fill=black,draw=black](7.05,0.35)circle [radius=0.01];
\filldraw[fill=black,draw=black](7.05,0.25)circle [radius=0.01];
\filldraw[fill=black,draw=black](7.05,0.15)circle [radius=0.01];

\draw[thick] (6.99,-0.11) circle [radius=(0.001)^(0.5)]; 
\draw[thick] (6.99,-0.51) circle [radius=(0.001)^(0.5)]; 

\filldraw[fill=black,draw=black](7.05,-0.35)circle [radius=0.01];
\filldraw[fill=black,draw=black](7.05,-0.25)circle [radius=0.01];
\filldraw[fill=black,draw=black](7.05,-0.15)circle [radius=0.01];

\draw[thick] (7.51,-0.02) circle [radius=(0.001)^(0.5)]; 
\draw[thick] (8.5,-0.02) circle [radius=(0.001)^(0.5)]; 
\draw[thick] (9.5,-0.02) circle [radius=(0.001)^(0.5)]; 
\draw[thick] (10.49,-0.02) circle [radius=(0.001)^(0.5)];

\node at (9,-0.5){$G^{*}$};

\node at (5.25,-1.5){$Figure\ 8: The\ graphs\ G\ and\ G^{*}\ in\ Lemma\ 11.$};
\end{tikzpicture}
\end{center}

\begin{proof}
Let $x$ be Perron eigenvector of $G^{*}$. We know that $x_{a_{1}}=\cdots=x_{a_{p}}=x_{b_{1}}=\cdots=x_{b_{q}}, x_{c_{1}}=\cdots=x_{c_{r}}, x_{d_{1}}=\cdots=x_{d_{r}}=x_{w}$. From Lemma \ref{lem9}, we know $x_{u}\geq x_{v_{2}}$, so $x_{a_{1}}\geq x_{d_{1}}> x_{v_{2}}$, using which we can get the following:
\begin{eqnarray*}
\rho(G^{*})(2x_{d_{1}}-x_{a_{1}}) > 8x_{a_{1}}-4x_{d_{1}}-x_{v_{2}}\Rightarrow[\rho(G^{*})+2]\cdot(2x_{d_{1}}-x_{a_{1}}) > 6x_{a_{1}}-x_{v_{2}}>0,
\end{eqnarray*}
thus $$2x_{w}=2x_{d_{1}}>x_{a_{1}}=x_{b_{1}}>x_{u}.$$
\indent Next, we can determine the sign of $3\sum\limits_{i=1}^p{x_{a_{i}}}+3{x_{u}}+x_{v_{1}}+\sum\limits_{i=1}^r{x_{c_{i}}}-{x_{v_{2}}}-\sum\limits_{i=1}^r{x_{d_{i}}}-3x_{w}$.

\begin{eqnarray*}
& &\rho(G^{*})[3\sum\limits_{i=1}^p{x_{a_{i}}}+3x_{u}+x_{v_{1}}+\sum\limits_{i=1}^r{x_{c_{i}}}-{x_{v_{2}}}-\sum\limits_{i=1}^r{x_{d_{i}}}-3x_{w}]\nonumber \\
&=& (6p-r-16) p x_{a_{1}}+(6p-r-10) q x_{b_{1}}+(3p-r-10) x_{u} +(9p-r-6) r x_{c_{1}} \nonumber \\
& &+(6p-r-4) x_{v_{1}}+(12p+r+6) r x_{d_{1}}+(9p+r+4) x_{v_{2}}+(12p+r+10) x_{w} \nonumber \\
\Rightarrow& & [\rho(G^{*})+r]\cdot[3\sum\limits_{i=1}^p{x_{a_{i}}}+3x_{u}+x_{v_{1}}+\sum\limits_{i=1}^r{x_{c_{i}}}-x_{v_{2}}-\sum\limits_{i=1}^r{x_{d_{i}}}-3x_{w}] \nonumber \\
&=& (6p+2r-16)px_{a_{1}}+(6p-r-10)qx_{b_{1}}+(3p+2r-10)x_{u}+(9p-6)rx_{c_{1}}  \nonumber \\
& &+(6p-4)x_{v_{1}}+(12p+6)rx_{d_{1}}+(9p+4)x_{v_{2}}+(12p-2r+10)x_{w}  \nonumber \\
&=&(6p+r-7)px_{a_{1}}+(6p-6)qx_{b_{1}}+(3p+2r-5)x_{u}+(6p-6)rx_{c_{1}} \nonumber \\
& &+(4p-4)x_{v_{1}}+(12p+4)rx_{d_{1}}+(9p+4)x_{v_{2}}+(12p+10)x_{w} \nonumber \\
& &+[(rpx_{a_{1}}-rqx_{b_{1}})-9px_{a_{1}}-4qx_{b_{1}}-5x_{u}+3prx_{c_{1}}+2px_{v_{1}}+(2rx_{d_{1}}-2rx_{w})] \nonumber \\
&\geq&(12p+4)rx_{d_{1}}+(9p+4)x_{v_{2}}+(12p+10)x_{w}-9px_{a_{1}}-4qx_{b_{1}}-5x_{u}+3prx_{c_{1}}+2px_{v_{1}} \nonumber \\
&=&(12prx_{d_{1}}+6px_{w}-9px_{a_{1}})+(4rx_{d_{1}}+4px_{w}-4qx_{b_{1}}) \nonumber \\
& &+(10x_{w}-5x_{u})+(9p+4)x_{v_{2}}+2px_{w}+3prx_{c_{1}}+2px_{v_{1}} > 0  \nonumber \\
\Rightarrow& &3\sum\limits_{i=1}^p{x_{a_{i}}}+3x_{u}+x_{v_{1}}+\sum\limits_{i=1}^r{x_{c_{i}}}-x_{v_{2}}-\sum\limits_{i=1}^r{x_{d_{i}}}-3x_{w}>0.
\end{eqnarray*}
\indent Using the above results, we can easily determine the sign of $\rho(G)-\rho(G^{*}).$
\begin{eqnarray*} 
\frac{1}{2}[\rho(G)-\rho(G^{*})]&\geq&\frac{1}{2}x^{t}[D(G)-D(G^{*})]x \nonumber \\
&=&(\sum\limits_{i=1}^q{x_{b_{i}}})\cdot [3\sum\limits_{i=1}^p{x_{a_{i}}}+3{x_{u}}+x_{v_{1}}+\sum\limits_{i=1}^r{x_{c_{i}}}-{x_{v_{2}}}-\sum\limits_{i=1}^r{x_{d_{i}}}-3x_{w}] \nonumber \\
&>&0\Rightarrow \rho(G)>\rho(G^{*}).
\end{eqnarray*}
\end{proof}

\begin{theorem}\label{lem12}
  
Let $T$ be a caterpillar graph which satisfies following conditions:
\begin{enumerate}
\item maximum degree is $\varDelta$ ($\varDelta \geq 4$). 
\item $|G{v_{2}}|=|v_{2(1)}|= m_{1}$.
\item $|G{v_{3}}|=|v_{3(1)}|=|G{v_{4}}|=|v_{4(1)}|= \varDelta -2$.
\item $|G{v_{5}}|=|v_{5(1)}|= m_{2}$.
\item $\varDelta -2> m_{1}\geq m_{2}$.
\item $ m_{1}+ m_{2}+2 > \varDelta -1.$
\end{enumerate}
Let $T^{*}=T-(\sum\limits_{i=1}^{m_{2}}{v_5 v_{5(1)(i)}}+v_{5}v_{6})+(\sum\limits_{i=1}^{m_{2}}{v_{3(1)(1)}v_{5(1)(i)}}+v_{3(1)(1)}v_{6})$. Then $\rho(T)>\rho(T^{*}).$
  
\end{theorem}

\begin{center}
\begin{tikzpicture}

\draw (0,0)node[below]{$v_{1}$} -- (3,0)node[below]{$v_{2}$} -- (6,0)node[below]{$v_{3}$} -- (9,0)node[below]{$v_{4}$} -- (12,0)node[below]{$v_{5}$} -- (15,0)node[below]{$v_{6}$};

\draw (2,1)node[above]{$v_{2(1)(1)}$} -- (3,0);
\filldraw[fill=black,draw=black](2.85,1)circle [radius=0.01];
\filldraw[fill=black,draw=black](3,1)circle [radius=0.01];
\filldraw[fill=black,draw=black](3.15,1)circle [radius=0.01]; 
\draw (4,1)node[above]{$v_{2(1)(m_{1})}$} -- (3,0);

\draw (5,1) -- (6,0);
\node at (4.7,0.7){$v_{3(1)(1)}$};
\filldraw[fill=black,draw=black](5.85,1)circle [radius=0.01];
\filldraw[fill=black,draw=black](6,1)circle [radius=0.01];
\filldraw[fill=black,draw=black](6.15,1)circle [radius=0.01]; 
\draw (7,1) -- (6,0);
\node at (6.7,1.3){$v_{3(1)(\varDelta -2)}$};

\draw (8,1)node[above]{$v_{4(1)(1)}$} -- (9,0);
\filldraw[fill=black,draw=black](8.85,1)circle [radius=0.01];
\filldraw[fill=black,draw=black](9,1)circle [radius=0.01];
\filldraw[fill=black,draw=black](9.15,1)circle [radius=0.01]; 
\draw (10,1)node[above]{$v_{4(1)(\varDelta -2)}$} -- (9,0);

\draw (11,1) -- (12,0);
\node at (10.7,0.7){$v_{5(1)(1)}$};
\filldraw[fill=black,draw=black](11.85,1)circle [radius=0.01];
\filldraw[fill=black,draw=black](12,1)circle [radius=0.01];
\filldraw[fill=black,draw=black](12.15,1)circle [radius=0.01]; 
\draw (13,1)node[above]{$v_{5(1)(m_{2})}$} -- (12,0);

\filldraw[fill=black,draw=black] (0,0) circle [radius=(0.0018)^(0.5)]; 
\filldraw[fill=black,draw=black] (3,0) circle [radius=(0.0018)^(0.5)]; 
\filldraw[fill=black,draw=black] (6,0) circle [radius=(0.0018)^(0.5)]; 
\filldraw[fill=black,draw=black] (9,0) circle [radius=(0.0018)^(0.5)]; 
\filldraw[fill=black,draw=black] (12,0) circle [radius=(0.0018)^(0.5)]; 
\filldraw[fill=black,draw=black] (15,0) circle [radius=(0.0018)^(0.5)]; 

\filldraw[fill=black,draw=black] (2,1) circle [radius=(0.0018)^(0.5)]; 
\filldraw[fill=black,draw=black] (4,1) circle [radius=(0.0018)^(0.5)]; 
\filldraw[fill=black,draw=black] (5,1) circle [radius=(0.0018)^(0.5)]; 
\filldraw[fill=black,draw=black] (7,1) circle [radius=(0.0018)^(0.5)]; 
\filldraw[fill=black,draw=black] (8,1) circle [radius=(0.0018)^(0.5)]; 
\filldraw[fill=black,draw=black] (10,1) circle [radius=(0.0018)^(0.5)]; 
\filldraw[fill=black,draw=black] (11,1) circle [radius=(0.0018)^(0.5)]; 
\filldraw[fill=black,draw=black] (13,1) circle [radius=(0.0018)^(0.5)];

\node at (6,-1){$T$};

\end{tikzpicture}
\end{center}

\begin{center}
\begin{tikzpicture}

\draw (0,0)node[below]{$v_{1}$} -- (3,0)node[below]{$v_{2}$} -- (6,0)node[below]{$v_{3}$} -- (9,0)node[below]{$v_{4}$} -- (12,0)node[below]{$v_{5}$};

\draw (2,1)node[above]{$v_{2(1)(1)}$} -- (3,0);
\filldraw[fill=black,draw=black](2.85,1)circle [radius=0.01];
\filldraw[fill=black,draw=black](3,1)circle [radius=0.01];
\filldraw[fill=black,draw=black](3.15,1)circle [radius=0.01]; 
\draw (4,1)node[above]{$v_{2(1)(m_{1})}$} -- (3,0);

\draw (5,1) -- (6,0);
\node at (4.7,0.7){$v_{3(1)(1)}$};
\filldraw[fill=black,draw=black](5.85,1)circle [radius=0.01];
\filldraw[fill=black,draw=black](6,1)circle [radius=0.01];
\filldraw[fill=black,draw=black](6.15,1)circle [radius=0.01]; 
\draw (7,1) -- (6,0);
\node at (6.7,1.3){$v_{3(1)(\varDelta -2)}$};

\draw (8,1)node[above]{$v_{4(1)(1)}$} -- (9,0);
\filldraw[fill=black,draw=black](8.85,1)circle [radius=0.01];
\filldraw[fill=black,draw=black](9,1)circle [radius=0.01];
\filldraw[fill=black,draw=black](9.15,1)circle [radius=0.01]; 
\draw (10,1)node[above]{$v_{4(1)(\varDelta -2)}$} -- (9,0);

\draw (4,2)node[above]{$v_{5(1)(1)}$} -- (5,1);
\filldraw[fill=black,draw=black](4.85,2)circle [radius=0.01];
\filldraw[fill=black,draw=black](5,2)circle [radius=0.01];
\filldraw[fill=black,draw=black](5.15,2)circle [radius=0.01]; 
\draw (6,2)node[above]{$v_{5(1)(m_{2})}$} -- (5,1);
\draw (6.5,2)node[below]{$v_{6}$} -- (5,1);

\filldraw[fill=black,draw=black] (0,0) circle [radius=(0.0018)^(0.5)]; 
\filldraw[fill=black,draw=black] (3,0) circle [radius=(0.0018)^(0.5)]; 
\filldraw[fill=black,draw=black] (6,0) circle [radius=(0.0018)^(0.5)]; 
\filldraw[fill=black,draw=black] (9,0) circle [radius=(0.0018)^(0.5)]; 
\filldraw[fill=black,draw=black] (12,0) circle [radius=(0.0018)^(0.5)]; 

\filldraw[fill=black,draw=black] (2,1) circle [radius=(0.0018)^(0.5)]; 
\filldraw[fill=black,draw=black] (4,1) circle [radius=(0.0018)^(0.5)]; 
\filldraw[fill=black,draw=black] (5,1) circle [radius=(0.0018)^(0.5)]; 
\filldraw[fill=black,draw=black] (7,1) circle [radius=(0.0018)^(0.5)]; 
\filldraw[fill=black,draw=black] (8,1) circle [radius=(0.0018)^(0.5)]; 
\filldraw[fill=black,draw=black] (10,1) circle [radius=(0.0018)^(0.5)]; 
\filldraw[fill=black,draw=black] (4,2) circle [radius=(0.0018)^(0.5)]; 
\filldraw[fill=black,draw=black] (6,2) circle [radius=(0.0018)^(0.5)]; 
\filldraw[fill=black,draw=black] (6.5,2) circle [radius=(0.0018)^(0.5)];

\node at (6,-1){$T^{*}$};

\node at (6,-2){$Figure\ 9: The\ graphs\ T\ and\ T^{*}\ in\ Lemma\ 12.$};
\end{tikzpicture}
\end{center}
\begin{proof}
Let $x$ be Perron eigenvector of $T^{*}.$\\
\indent From Lemma \ref{lem9} we can see that, for $i\in [2,\varDelta -2]$: 
\begin{eqnarray}
x_{v_{3(1)(i)}}> x_{v_{3(1)(1)}}\geq x_{v_{2}}> x_{v_{4}}.\label{12.1}
\end{eqnarray}
\indent Since $m_{1}+m_{2}+2\geq \varDelta$, then by Lemma \ref{lem1} and (\ref{12.1}), we can follow up that, for $i\in [2,\varDelta -2]$:
\begin{eqnarray}
\rho(T^{*})\cdot(x_{v_{4}}-x_{v_{3}})&>&[\sum\limits_{i=1}^{2}{x_{v_{i}}}+\sum\limits_{i=1}^{m_{1}}{x_{v_{2(1)(i)}}}+x_{v_{3(1)(1)}}+x_{v_{6}}+\sum\limits_{i=1}^{m_{2}}{x_{v_{5(1)(i)}}}] -[\sum\limits_{i=4}^{5}{x_{v_{i}}}+\sum\limits_{i=1}^{\varDelta -2}{x_{v_{4(1)(i)}}}] \nonumber \\
&>&0\Rightarrow x_{v_{4}}>x_{v_{3}}\Rightarrow x_{v_{3(1)(i)}}> x_{v_{3(1)(1)}}\geq x_{v_{2}}>x_{v_{4}}>x_{v_{3}}.\label{12.2}
\end{eqnarray}
%\indent Then we determine the sign of $3x_{v_{4}}-x_{v_{2}}$, $x_{v_{3}}+x_{v_{3(1)(1)}}+x_{v_{3(1)(2)}}-2x_{v_{5}}$ and $\frac{3}{2}x_{v_{3(1)(1)}}-x_{v_{4(1)(1)}}$.
\indent Then we can determine the sign of $x_{v_{3}}+x_{v_{3(1)(1)}}+x_{v_{3(1)(2)}}-2x_{v_{5}}$ and $\frac{3}{2}x_{v_{3(1)(1)}}-x_{v_{4(1)(1)}}$.\\
%{\bfseries Claim 3:} $3x_{v_{4}}>x_{v_{2}}$.
%\begin{eqnarray}
%& &\rho(T^{*})(3x_{v_{4}}-x_{v_{2}}) > 6x_{v_{2}}-2x_{v_{4}}\Rightarrow [\rho(T^{*})+1]\cdot(3x_{v_{4}}-x_{v_{2}}) > 0\Rightarrow 3x_{v_{4}}>x_{v_{2}}\label{12.3}
%\end{eqnarray}
\indent By (\ref{12.2}), Lemma \ref{lem1}, Lemma \ref{lem4} and Lemma \ref{lem9}, we have
\begin{eqnarray}
& &\rho(T^{*})\cdot[x_{v_{3}}+x_{v_{3(1)(1)}}+x_{v_{3(1)(2)}}-2x_{v_{5}}]\nonumber \\
&=& 8x_{v_{5}} +4\cdot \sum\limits_{i=1}^{\varDelta -2}{x_{v_{4(1)(i)}}} + 3x_{v_{4}} -2[\sum\limits_{i=1}^{m_{2}}{x_{v_{5(1)(i)}}} +x_{v_{6}}] -3x_{v_{3(1)(2)}}-3x_{v_{3(1)(1)}} -\sum\limits_{i=3}^{\varDelta -2}{x_{v_{3(1)(i)}}} \nonumber \\
& &-2x_{v_{3}} -2x_{v_{2}} +0\cdot[x_{v_{1}} +\sum\limits_{i=1}^{m_{1}}{x_{v_{2(1)(i)}}}]\nonumber \\
&=&\{[\sum\limits_{i=1}^{\varDelta -2}{x_{v_{4(1)(i)}}} +x_{v_{5}}] - [\sum\limits_{i=1}^{\varDelta -2}{x_{v_{3(1)(i)}}} +x_{v_{2}}]\}+(x_{v_{4}}-x_{v_{3}}) + (x_{v_{5}}-x_{v_{3}}) +2[x_{v_{5}}-x_{v_{3(1)(2)}}]\nonumber \\
& & +2\{[x_{v_{4}} +\sum\limits_{i=1}^{\varDelta -2}{x_{v_{4(1)(i)}}} +x_{v_{5}}] - [x_{v_{3(1)(1)}}+ \sum\limits_{i=1}^{m_{2}}{x_{v_{5(1)(i)}}} +x_{v_{6}}]\}+ [\sum\limits_{i=1}^{\varDelta -2}{x_{v_{4(1)(i)}}} +x_{v_{5}} -x_{v_{2}}] \nonumber \\
& &+x_{v_{5}}\nonumber \\
&>&0\Rightarrow x_{v_{3}}+x_{v_{3(1)(1)}}+x_{v_{3(1)(2)}}>2x_{v_{5}}.\label{12.4}
\end{eqnarray}
\indent Since $m_{1}+ m_{2}+2> \varDelta -1$, and by (\ref{12.2}), Lemma \ref{lem1}, Lemma \ref{lem4}, we have
\begin{eqnarray}
& &\rho(T^{*})[\frac{3}{2}x_{v_{4}}-x_{v_{4(1)(1)}}]\nonumber \\
&=&\frac{1}{2}[x_{v_{1}}+\sum\limits_{i=1}^{m_{1}}{x_{v_{2(1)(i)}}}+x_{v_{6}}+\sum\limits_{i=1}^{m_{2}}{x_{v_{5(1)(i)}}}] +0\cdot[x_{v_{2}} +\sum\limits_{i=1}^{\varDelta -2}{x_{v_{3(1)(i)}}}] -\frac{1}{2}[x_{v_{3}}+ \sum\limits_{i=2}^{\varDelta -2}{x_{v_{4(1)(i)}}}+x_{v_{5}}] \nonumber \\
& &-x_{v_{4}}+\frac{3}{2}x_{v_{4(1)(1)}}\nonumber \\
&=&\frac{1}{2}\{[\sum\limits_{i=1}^{m_{1}}{x_{v_{2(1)(i)}}} + x_{v_{6}}+\sum\limits_{i=1}^{m_{2}}{x_{v_{5(1)(i)}}}] - [\sum\limits_{i=1}^{\varDelta -2}{x_{v_{4(1)(i)}}}+ x_{v_{5}}]\} +\frac{1}{2}(x_{v_{1}}-x_{v_{3}}) + [x_{v_{4(1)(1)}}-x_{v_{4}}] \nonumber \\
& &+ \frac{1}{2}x_{v_{1}} + \frac{1}{2}x_{v_{4(1)(1)}}\nonumber \\
&>&0\Rightarrow \frac{3}{2}x_{v_{3(1)(1)}}>\frac{3}{2}x_{v_{4}}>x_{v_{4(1)(1)}}\label{12.5}
\end{eqnarray}
\indent As we know that  $\varDelta -2> m_{1}\geq m_{2}$ and $m_{1}+ m_{2}+2> \varDelta -1$, which implies that
\begin{eqnarray}
2m_{1}\geq m_{1}+ m_{2}\geq \varDelta -2 \Rightarrow m_{1}\geq \lceil \frac{\varDelta -2}{2}\rceil.\label{12.6}
\end{eqnarray}
\indent So by Lemma \ref{lem1}, (\ref{12.2}), (\ref{12.5}), (\ref{12.6}), 
\begin{eqnarray}
& &[2x_{v_{3(1)(1)}}+\sum\limits_{i=3}^{\varDelta -2}{x_{v_{3(1)(i)}}}+x_{v_{2}}+x_{v_{1}}+\sum\limits_{i=1}^{m_{1}}{x_{v_{2(1)(i)}}}]- [x_{v_{4}}+x_{v_{5}}+\sum\limits_{i=1}^{\varDelta -2}{x_{v_{4(1)(i)}}}]  \nonumber \\
&=&(x_{v_{2}}-x_{v_{4}})+(x_{v_{1}}-x_{v_{5}}) +[2x_{v_{3(1)(1)}}+\sum\limits_{i=3}^{\varDelta -2}{x_{v_{3(1)(i)}}}+\sum\limits_{i=1}^{m_{1}}{x_{v_{2(1)(i)}}}-\sum\limits_{i=1}^{\varDelta -2}{x_{v_{4(1)(i)}}}]   \nonumber \\
&>& (\varDelta -2)\cdot x_{v_{3(1)(1)}} +\lceil \frac{\varDelta -2}{2}\rceil \cdot x_{v_{2(1)(1)}} -(\varDelta -2)\cdot x_{v_{4(1)(1)}} \nonumber \\
&>&(\varDelta -2)\cdot x_{v_{3(1)(1)}} - \lfloor \frac{\varDelta -2}{2} \rfloor \cdot x_{v_{4(1)(1)}} \geq \lfloor \frac{\varDelta -2}{2} \rfloor \cdot [2x_{v_{3(1)(1)}}-x_{v_{4(1)(1)}}]  > 0. \label{12.7}
\end{eqnarray}
\indent By (\ref{12.4}) and (\ref{12.7}), we can determine the sign of $\rho(T)-\rho(T^{*})$.
\begin{eqnarray*}
& &\frac{1}{2}[\rho(T)-\rho(T^{*})]\nonumber \\
&\geq&\frac{1}{2}x^{t}[D(T)-D(T^{*})]x \nonumber \\
&=&[x_{v_{6}}+\sum\limits_{i=1}^{m_{2}}{x_{v_{5(1)(i)}}}] \{3x_{v_{3(1)(1)}} +[\sum\limits_{i=1}^{3}{x_{v_{i}}}+\sum\limits_{i=2}^{\varDelta -2}{x_{v_{3(1)(i)}}}+\sum\limits_{i=1}^{m_{1}}{x_{v_{2(1)(i)}}}]-[x_{v_{4}}+\sum\limits_{i=1}^{\varDelta -2}{x_{v_{4(1)(i)}}}]-3x_{v_{5}}\}  \nonumber\\
&=&[x_{v_{6}}+\sum\limits_{i=1}^{m_{2}}{x_{v_{5(1)(i)}}}]\{x_{v_{3}}+x_{v_{3(1)(1)}}+x_{v_{3(1)(2)}}+1\cdot [2x_{v_{3(1)(1)}}+\sum\limits_{i=3}^{\varDelta -2}{x_{v_{3(1)(i)}}}+x_{v_{2}}+x_{v_{1}}+\sum\limits_{i=1}^{m_{1}}{x_{v_{2(1)(i)}}}]\nonumber\\
& &\ \ \ \ \ \ \ \ \ \ \ \ \ \ \ \ \ \ \ \ \ \ \ \ \ \ -1\cdot [x_{v_{4}}+x_{v_{5}}+\sum\limits_{i=1}^{\varDelta -2}{x_{v_{4(1)(i)}}}]-2x_{v_{5}}\} \nonumber \\
&>& 0\Rightarrow \rho(T)>\rho(T^{*}).
\end{eqnarray*}
\end{proof}

\begin{theorem}\label{lem13}
  
Let $P=v_{1}v_{2}\cdots v_{q}v_{q+1}$ be a path of tree $T$ ($q\geq 4$), where $T$ satisfies the following conditions:
\begin{enumerate}
\item $|G_{v_{2a}}|+|G_{v_{2b}}|=|G_{v_{2}}|=|v_{2(1)}|.$
\item $\sum\limits_{i=4}^{q}{|G_{v_{i}}^{p}|}\geq|G_{v_{2b}}|$, $|G_{v_{2a}}|\in [1,|G_{v_{2}}|].$
\end{enumerate}
Let $T^{*}=T-\sum\limits_{i=1}^{|G_{v_{2a}}|}{v_{2}v_{2(1)(i)}}+\sum\limits_{i=1}^{|G_{v_{2a}}|}{v_{3}v_{2(1)(i)}}$, then $\rho(T)>\rho(T^{*})$.  
\end{theorem}

\begin{center}
\begin{tikzpicture}

\draw (0,0)node[below]{$v_{1}$} -- (2,0)node[below]{$v_{2}$} -- (4,0)node[below]{$v_{3}$} -- (6,0)node[below]{$v_{4}$} -- (8,0)node[below]{$v_{5}$};

\filldraw[fill=black,draw=black](8.7,0) circle [radius=0.01];
\filldraw[fill=black,draw=black](9,0) circle [radius=0.01];
\filldraw[fill=black,draw=black](9.3,0) circle [radius=0.01];

\draw (10,0)node[below]{$v_{q}$} -- (12,0)node[below]{$v_{q+1}$};

\draw (2,0.6) ellipse[x radius=0.3,y radius=0.6];
\node at (2,0.6){$G_{v_{2}}$};
\draw (4,0.6) ellipse[x radius=0.3,y radius=0.6];
\node at (4,0.6){$G_{v_{3}}$};
\draw (6,0.6) ellipse[x radius=0.3,y radius=0.6];
\node at (6,0.6){$G_{v_{4}}$};
\draw (8,0.6) ellipse[x radius=0.3,y radius=0.6];
\node at (8,0.6){$G_{v_{5}}$};
\draw (10,0.6) ellipse[x radius=0.3,y radius=0.6];
\node at (10,0.6){$G_{v_{q}}$};

\filldraw[fill=black,draw=black] (0,0) circle [radius=(0.0018)^(0.5)]; 
\filldraw[fill=black,draw=black] (2,0) circle [radius=(0.0018)^(0.5)]; 
\filldraw[fill=black,draw=black] (4,0) circle [radius=(0.0018)^(0.5)]; 
\filldraw[fill=black,draw=black] (6,0) circle [radius=(0.0018)^(0.5)]; 
\filldraw[fill=black,draw=black] (8,0) circle [radius=(0.0018)^(0.5)]; 
\filldraw[fill=black,draw=black] (10,0) circle [radius=(0.0018)^(0.5)]; 
\filldraw[fill=black,draw=black] (12,0) circle [radius=(0.0018)^(0.5)];

\node at (6,-1){$T$};
\end{tikzpicture}
\end{center}

\begin{center}
\begin{tikzpicture}

\draw (0,0)node[below]{$v_{1}$} -- (2,0)node[below]{$v_{2}$} -- (4,0)node[below]{$v_{3}$} -- (6,0)node[below]{$v_{4}$} -- (8,0)node[below]{$v_{5}$};

\filldraw[fill=black,draw=black](8.7,0) circle [radius=0.01];
\filldraw[fill=black,draw=black](9,0) circle [radius=0.01];
\filldraw[fill=black,draw=black](9.3,0) circle [radius=0.01];

\draw (10,0)node[below]{$v_{q}$} -- (12,0)node[below]{$v_{q+1}$};

\draw (2,0.6) ellipse[x radius=0.4,y radius=0.6];
\node at (2,0.6){$G_{v_{2b}}$};
\draw (4,0.6) ellipse[x radius=0.6,y radius=0.6];
\node at (4,0.8){$G_{v_{3}}$};
\node at (4,0.4){$\cup G_{v_{2a}}$};
\draw (6,0.6) ellipse[x radius=0.3,y radius=0.6];
\node at (6,0.6){$G_{v_{4}}$};
\draw (8,0.6) ellipse[x radius=0.3,y radius=0.6];
\node at (8,0.6){$G_{v_{5}}$};
\draw (10,0.6) ellipse[x radius=0.3,y radius=0.6];
\node at (10,0.6){$G_{v_{q}}$};

\filldraw[fill=black,draw=black] (0,0) circle [radius=(0.0018)^(0.5)]; 
\filldraw[fill=black,draw=black] (2,0) circle [radius=(0.0018)^(0.5)]; 
\filldraw[fill=black,draw=black] (4,0) circle [radius=(0.0018)^(0.5)]; 
\filldraw[fill=black,draw=black] (6,0) circle [radius=(0.0018)^(0.5)]; 
\filldraw[fill=black,draw=black] (8,0) circle [radius=(0.0018)^(0.5)]; 
\filldraw[fill=black,draw=black] (10,0) circle [radius=(0.0018)^(0.5)]; 
\filldraw[fill=black,draw=black] (12,0) circle [radius=(0.0018)^(0.5)];

\node at (6,-1){$T^{*}$};

\node at (6,-2){$Figure\ 10: The\ graphs\ T\ and\ T^{*}\ in\ Lemma\ 13.$};
\end{tikzpicture}
\end{center}

\begin{proof}
Let $x$ be Perron eigenvector of $T^{*}$. If 
\begin{equation}\label{13.1}
\sum\limits_{i=3}^{q}{S(G_{v_{i}})} + \sum\limits_{i=3}^{q+1}{x_{v_{i}}} - x_{v_{1}} - x_{v_{2}} - S(G_{v_{2b}}) \leq 0,
\end{equation}
then we can say that
\begin{equation}\label{13.2}
\sum\limits_{i=4}^{q}{S(G_{v_{i}})} + \sum\limits_{i=4}^{q+1}{x_{v_{i}}} - x_{v_{1}} - x_{v_{2}} - S(G_{v_{2b}}) < 0.
\end{equation}
\indent Using inequality (\ref{13.2}) we get
\begin{equation}\label{13.3}
\rho(T^{*}) (x_{v_{2}}-x_{v_{4}})=2[\sum\limits_{i=4}^{q}{S(G_{v_{i}})}+\sum\limits_{i=4}^{q+1}{x_{v_{i}}}-x_{v_{1}}-x_{v_{2}}-S(G_{v_{2b}})]<0 \Rightarrow x_{v_{2}}< x_{v_{4}}. 
\end{equation}
\indent Now since we have
\begin{eqnarray}
\rho(T^{*})\cdot (x_{v_{4}}-x_{v_{5}}) &=&\sum\limits_{i=5}^{q}{S(G_{v_{i}})}+\sum\limits_{i=5}^{q+1}{x_{v_{i}}}-\sum\limits_{i=1}^4{x_{v_{i}}}-S(G_{v_{2a}})-S(G_{v_{2b}})-S(G_{v_{3}})-S(G_{v_{4}})  \nonumber \\
&<&\sum\limits_{i=4}^{q}{S(G_{v_{i}})}+\sum\limits_{i=4}^{q+1}{x_{v_{i}}}-x_{v_{1}}-x_{v_{2}}-S(G_{v_{2b}}) < 0\Rightarrow x_{v_{4}}<x_{v_{5}},\label{13.4}
\end{eqnarray}
by (\ref{13.3}), (\ref{13.4}) and Lemma \ref{lem5} we conclude that
\begin{equation}\label{13.5}
x_{v_{2}}<x_{v_{4}}<x_{v_{5}}<x_{v_{6}}<\ldots<x_{v_{q}}<x_{v_{q+1}}.
\end{equation}
\indent And by (\ref{13.5}) and Lemma \ref{lem1} we have
\begin{equation}\label{13.6}
x_{v_{q+1}}> x_{v_{1}}.
\end{equation}
\indent Note that $\sum\limits_{i=4}^{q}{|G_{v_{i}}^{p}|}\geq|G_{v_{2b}}|$. So by (\ref{13.5}), Lemma \ref{lem1} and Lemma \ref{lem5}, we have
\begin{equation}\label{13.6}
\sum\limits_{i=4}^{q}{S(G_{v_{i}})} \geq \sum\limits_{i=4}^{q}{S(G_{v_{i}}^{p})}>S(G_{v_{2b}}).
\end{equation}
\indent Combining all the above results, we get
\begin{eqnarray*}
\sum\limits_{i=3}^{q}{S(G_{v_{i}})}+\sum\limits_{i=3}^{q+1}{x_{v_{i}}}-\sum\limits_{i=1}^{2}{x_{v_{i}}}-S(G_{v_{2b}}) > \sum\limits_{i=4}^{q}{S(G_{v_{i}})}+\sum\limits_{i=4}^{q+1}{x_{v_{i}}}-\sum\limits_{i=1}^{2}{x_{v_{i}}}-S(G_{v_{2b}}) > 0, 
\end{eqnarray*}
a contradiction. So (\ref{13.1}) doesn't hold. Then we can determine the sign of $\rho(T)-\rho(T^{*})$.
\begin{eqnarray*}
\frac{1}{2}[\rho(T)-\rho(T^{*})]&\geq&\frac{1}{2} x^{t} [D(T)-D(T^{*})]x \nonumber \\
&=&S(G_{v_{2a}})\cdot[\sum\limits_{i=3}^{q}{S(G_{v_{i}})}+\sum\limits_{i=3}^{q+1}{x_{v_{i}}}-\sum\limits_{i=1}^{2}{x_{v_{i}}}-S(G_{v_{2b}})] > 0 \Rightarrow \rho(T)>\rho(T^{*}).
\end{eqnarray*}
\end{proof}

\begin{theorem}\label{lem14}  
Let $P=v_{1}v_{2}\cdots v_{q}v_{q+1}$ be a path of tree $T$ ($q\geq 5$), where $T$ satisfies following conditions: 
\begin{enumerate}
\item maximum degree is $\varDelta$.
\item $|G_{v_{2a}}|+|G_{v_{2b}}|=|G_{v_{2}}|=|v_{2(1)}|$, $|G_{v_{2a}}|=1$.
\item $|G_{v_{2a}}|+|G_{v_{4(1)}}|\leq \varDelta -2$.
\item $\sum\limits_{i=5}^{q}{|G_{v_{i}}^{p}|}\geq|G_{v_{2b}}|$.
\end{enumerate}
Let $T^{*}=T-\sum\limits_{i=1}^{|G_{v_{2a}}|}{v_{2}v_{2(1)(i)}}+\sum\limits_{i=1}^{|G_{v_{2a}}|}{v_{4}v_{2(1)(i)}}$, then $\rho(T)>\rho(T^{*})$.
\end{theorem}

\begin{center}
\begin{tikzpicture}
\draw (0,0)node[below]{$v_{1}$} -- (2,0)node[below]{$v_{2}$} -- (4,0)node[below]{$v_{3}$} -- (6,0)node[below]{$v_{4}$} -- (8,0)node[below]{$v_{5}$} -- (10,0)node[below]{$v_{6}$};

\filldraw[fill=black,draw=black](10.7,0) circle [radius=0.01];
\filldraw[fill=black,draw=black](11,0) circle [radius=0.01];
\filldraw[fill=black,draw=black](11.3,0) circle [radius=0.01];

\draw (12,0)node[below]{$v_{q}$} -- (14,0)node[below]{$v_{q+1}$};

\draw (2,0.6) ellipse[x radius=0.3,y radius=0.6];
\node at (2,0.6){$G_{v_{2}}$};
\draw (4,0.6) ellipse[x radius=0.3,y radius=0.6];
\node at (4,0.6){$G_{v_{3}}$};
\draw (6,0.6) ellipse[x radius=0.3,y radius=0.6];
\node at (6,0.6){$G_{v_{4}}$};
\draw (8,0.6) ellipse[x radius=0.3,y radius=0.6];
\node at (8,0.6){$G_{v_{5}}$};
\draw (10,0.6) ellipse[x radius=0.3,y radius=0.6];
\node at (10,0.6){$G_{v_{6}}$};
\draw (12,0.6) ellipse[x radius=0.3,y radius=0.6];
\node at (12,0.6){$G_{v_{q}}$};

\filldraw[fill=black,draw=black] (0,0) circle [radius=(0.0018)^(0.5)]; 
\filldraw[fill=black,draw=black] (2,0) circle [radius=(0.0018)^(0.5)]; 
\filldraw[fill=black,draw=black] (4,0) circle [radius=(0.0018)^(0.5)]; 
\filldraw[fill=black,draw=black] (6,0) circle [radius=(0.0018)^(0.5)]; 
\filldraw[fill=black,draw=black] (8,0) circle [radius=(0.0018)^(0.5)]; 
\filldraw[fill=black,draw=black] (10,0) circle [radius=(0.0018)^(0.5)]; 
\filldraw[fill=black,draw=black] (12,0) circle [radius=(0.0018)^(0.5)]; 
\filldraw[fill=black,draw=black] (14,0) circle [radius=(0.0018)^(0.5)];

\node at (7,-1){$T$};
\end{tikzpicture}
\end{center}

\begin{center}
\begin{tikzpicture}

\draw (0,0)node[below]{$v_{1}$} -- (2,0)node[below]{$v_{2}$} -- (4,0)node[below]{$v_{3}$} -- (6,0)node[below]{$v_{4}$} -- (8,0)node[below]{$v_{5}$} -- (10,0)node[below]{$v_{6}$};

\filldraw[fill=black,draw=black](10.7,0) circle [radius=0.01];
\filldraw[fill=black,draw=black](11,0) circle [radius=0.01];
\filldraw[fill=black,draw=black](11.3,0) circle [radius=0.01];

\draw (12,0)node[below]{$v_{q}$} -- (14,0)node[below]{$v_{q+1}$};

\draw (2,0.6) ellipse[x radius=0.4,y radius=0.6];
\node at (2,0.6){$G_{v_{2b}}$};
\draw (4,0.6) ellipse[x radius=0.3,y radius=0.6];
\node at (4,0.6){$G_{v_{3}}$};
\draw (6,0.6) ellipse[x radius=0.6,y radius=0.6];
\node at (6,0.8){$G_{v_{4}}$};
\node at (6,0.4){$\cup G_{v_{2a}}$};
\draw (8,0.6) ellipse[x radius=0.3,y radius=0.6];
\node at (8,0.6){$G_{v_{5}}$};
\draw (10,0.6) ellipse[x radius=0.3,y radius=0.6];
\node at (10,0.6){$G_{v_{6}}$};
\draw (12,0.6) ellipse[x radius=0.3,y radius=0.6];
\node at (12,0.6){$G_{v_{q}}$};

\filldraw[fill=black,draw=black] (0,0) circle [radius=(0.0018)^(0.5)]; 
\filldraw[fill=black,draw=black] (2,0) circle [radius=(0.0018)^(0.5)]; 
\filldraw[fill=black,draw=black] (4,0) circle [radius=(0.0018)^(0.5)]; 
\filldraw[fill=black,draw=black] (6,0) circle [radius=(0.0018)^(0.5)]; 
\filldraw[fill=black,draw=black] (8,0) circle [radius=(0.0018)^(0.5)]; 
\filldraw[fill=black,draw=black] (10,0) circle [radius=(0.0018)^(0.5)]; 
\filldraw[fill=black,draw=black] (12,0) circle [radius=(0.0018)^(0.5)]; 
\filldraw[fill=black,draw=black] (14,0) circle [radius=(0.0018)^(0.5)];

\node at (7,-1){$T^{*}$};

\node at (7,-2){$Figure\ 11: The\ graphs\ T\ and\ T^{*}\ in\ Lemma\ 14.$};
\end{tikzpicture}
\end{center}

\begin{proof}
Let $x$ be Perron eigenvector of $T^{*}$. If 
\begin{equation} \label{14.1}
\sum\limits_{i=4}^q{S(G_{v_{i}})}+\sum\limits_{i=4}^{q+1}{x_{v_{i}}}-S(G_{v_{2b}})-\sum\limits_{i=1}^2{x_{v_{i}}} \leq 0,
\end{equation}
then using inequality (\ref{14.1}), we get
\begin{eqnarray}
\rho(T^{*})\cdot(x_{v_{4}}-x_{v_{5}}) &=& \sum\limits_{i=5}^q{S(G_{v_{i}})}+\sum\limits_{i=5}^{q+1}{x_{v_{i}}}-S(G_{v_{2a}})-S(G_{v_{2b}})-\sum\limits_{i=3}^4{S(G_{v_{i}})}-\sum\limits_{i=1}^4{x_{v_{i}}} \nonumber \\
&<& 0 \Rightarrow x_{v_{4}}<x_{v_{5}}.\label{14.2}
\end{eqnarray}
%which implies that $x_{v_{4}}<x_{v_{5}}$.\\
\indent Since
\begin{eqnarray}
\rho(T^{*})\cdot (S(G_{v_{2a}})-2x_{v_{4}})<x_{v_{4}}-2S(G_{v_{2a}}) \Rightarrow[\rho(T^{*})+1] (S(G_{v_{2a}})-2x_{v_{4}})<0 \Rightarrow S(G_{v_{2a}})<2x_{v_{4}},\label{14.3}
\end{eqnarray}
\noindent by (\ref{14.3}) we have
\begin{eqnarray}
& &\rho(T^{*})\cdot (x_{v_{2}}-x_{v_{5}})\nonumber \\
&=&3[\sum\limits_{i=5}^q{S(G_{v_{i}})}+\sum\limits_{i=5}^{q+1}{x_{v_{i}}}-S(G_{v_{2b}})-\sum\limits_{i=1}^2{x_{v_{i}}}] +[S(G_{v_{2a}})+S(G_{v_{4}})+x_{v_{4}}-S(G_{v_{3}})-x_{v_{3}}] \nonumber \\
&=&3[\sum\limits_{i=4}^q{S(G_{v_{i}})}+\sum\limits_{i=4}^{q+1}{x_{v_{i}}}-S(G_{v_{2b}})-\sum\limits_{i=1}^2{x_{v_{i}}}]+[S(G_{v_{2a}})-2S(G_{v_{4}})-2x_{v_{4}}-S(G_{v_{3}})-x_{v_{3}}] \nonumber \\
&<&S(G_{v_{2a}})-2x_{v_{4}} < 0 \Rightarrow x_{v_{2}} < x_{v_{5}}.\label{14.4}
\end{eqnarray}
\indent By (\ref{14.2}), (\ref{14.4}) and Lemma \ref{lem5} we conclude that
%\indent So, we get
\begin{eqnarray}
x_{v_{2}} < x_{v_{5}} < x_{v_{6}} < \cdots < x_{v_{q}} < x_{v_{q+1}}.\label{14.5}
\end{eqnarray}
\indent And by (\ref{14.5}) and Lemma \ref{lem1} we have
\begin{equation}\label{14.6}
x_{v_{q+1}}> x_{v_{1}}.
\end{equation}
\indent Note that $\sum\limits_{i=5}^{q}{|G_{v_{i}}^{p}|}\geq |G_{v_{2b}}|$. So by (\ref{14.5}), Lemma \ref{lem1} and Lemma \ref{lem5} we have
\begin{equation} \label{14.7}
\sum\limits_{i=5}^{q}{S(G_{v_{i}})} \geq \sum\limits_{i=5}^{q}{S(G_{v_{i}}^{p})} > S(G_{v_{2b}}).
\end{equation}
%\indent It's easy to conclude from above that $x_{v_{q}} > x_{v_{2}}$ and $x_{v_{q+1}} > x_{v_{1}}$.\\
\indent Combining the results above, we know
$$\sum\limits_{i=4}^q{S(G_{v_{i}})} + \sum\limits_{i=4}^{q+1}{x_{v_{i}}} - S(G_{v_{2b}}) - \sum\limits_{i=1}^2{x_{v_{i}}} > 0,$$
a contradiction. So (\ref{14.1}) doesn't hold. Then we can determine the sign of $\rho(T)-\rho(T^{*})$.
\begin{eqnarray*}
\frac{1}{2}[\rho(T)-\rho(T^{*})]&\geq&\frac{1}{2}x^{t}[D(T)-D(T^{*})]x \nonumber \\
&=&S(G_{v_{2a}})\cdot2\cdot[\sum\limits_{i=4}^q{S(G_{v_{i}})}+\sum\limits_{i=4}^{q+1}{x_{v_{i}}}-S(G_{v_{2b}})-\sum\limits_{i=1}^2{x_{v_{i}}}] > 0\Rightarrow \rho(T)>\rho(T^{*}).
\end{eqnarray*}
%so $\rho(T)>\rho(T^{*})$.\\
\indent If $|G_{v_{2a}}|+|G_{v_{4(1)}}|<\varDelta -2$ and $|\{v_{1}\}| + |G_{v_{2}}| + |G_{v_{4(1)}}| > \varDelta -2$, then repeatedly using the approach in the proof above, we can get $T^{**}$ with $|G_{v_{2a}}|+|G_{v_{4(1)}}|=\varDelta -2$ and $\rho(T)>\rho(T^{**})$. 
\end{proof}

\begin{theorem}\label{lem15}  
Let $P=v_{1}v_{2}\cdots v_{q}v_{q+1}$ be a path of tree $T$ ($q\geq 5$). Let $G_{v_{2}}^{''}=G_{v_{2}}\cup \{v_{1}\}\cup  v_{1}v_{4(1)(m_{3}+1)}$. $T$ satisfies following conditions:
\begin{enumerate}
\item $|G_{v_{2}}|=|v_{2(1)}|=m_{1}\leq \varDelta -2$.
\item $\sum\limits_{i=1}^{m_{3}}{|G_{v_{4(1)(i)}}^{p}|}+\sum\limits_{i=5}^q{|G_{v_{i}}^{p}|}\geq |v_{3(1)}|=\varDelta -2$. 
\item $|G_{v_{3(1)(i)}}| = |v_{3(2)(i)}| = \varDelta -2$, $i\in [1,m_{2}]$.
\item $deg(v_{3(1)(i)})=1$, $i\in [m_{2}+1,\varDelta-2]$.
\item $(m_{2}+1)+m_{3}\leq \varDelta -2$.
\end{enumerate}
Let $T^{*}=T-(\sum\limits_{i=1}^{m_{1}}{v_{2}v_{2(1)(i)}}+v_{2}v_{1}+\sum\limits_{i=1}^{m_{2}}{v_{3}v_{3(1)(i)}}) + (v_{4(1)(m_{3}+1)}v_{1}+\sum\limits_{i=1}^{m_{1}}{v_{4(1)(m_{3}+1)}v_{2(1)(i)}}+\sum\limits_{i=1}^{m_{2}}{v_{4(1)(m_{3}+i+1)}v_{3(1)(i)}})$, then $\rho(T)>\rho(T^{*})$.  
\end{theorem}

\begin{center}
\begin{tikzpicture}
\draw (0,0)node[below]{$v_{1}$} -- (2,0)node[below]{$v_{2}$} -- (4,0)node[above]{$v_{3}$} -- (7,0)node[below]{$v_{4}$} -- (10,0)node[above]{$v_{5}$};

\filldraw[fill=black,draw=black](10.7,0) circle [radius=0.01];
\filldraw[fill=black,draw=black](11,0) circle [radius=0.01];
\filldraw[fill=black,draw=black](11.3,0) circle [radius=0.01];

\draw (12,0)node[above]{$v_{q}$} -- (14,0)node[below]{$v_{q+1}$};

\draw (2,0.6) ellipse[x radius=0.3,y radius=0.6];
\node at (2,0.6){$G_{v_{2}}$};

\draw (1.25,-1)node[above left]{$v_{3(1)(1)}$} -- (4,0);
\filldraw[fill=black,draw=black] (1.25,-1) circle [radius=(0.0018)^(0.5)]; 
\draw (1.25,-1.3) ellipse[x radius=0.8,y radius=0.3];
\node at (1.25,-1.3){$G_{v_{3(1)(1)}}$};

\filldraw[fill=black,draw=black](1.75,-1) circle [radius=0.01];
\filldraw[fill=black,draw=black](2.05,-1) circle [radius=0.01];
\filldraw[fill=black,draw=black](2.35,-1) circle [radius=0.01];

\draw (3.25,-1) -- (4,0);
\node at (4.5,-0.75){$v_{3(1)(m_{2})}$};
\filldraw[fill=black,draw=black] (3.25,-1) circle [radius=(0.0018)^(0.5)]; 
\draw (3.25,-1.3) ellipse[x radius=1,y radius=0.3];
\node at (3.25,-1.3){$G_{v_{3(1)(m_{2})}}$};

\filldraw[fill=black,draw=black](6.5,-1) circle [radius=0.01];
\filldraw[fill=black,draw=black](6.85,-1) circle [radius=0.01];
\filldraw[fill=black,draw=black](7.15,-1) circle [radius=0.01];

\draw (6,-1)node[below]{$v_{3(1)(m_{2}+1)}$} -- (4,0);
\filldraw[fill=black,draw=black] (6,-1) circle [radius=(0.0018)^(0.5)];

\draw (8,-1)node[below right]{$v_{3(1)(\varDelta -2)}$} -- (4,0);
\filldraw[fill=black,draw=black] (8,-1) circle [radius=(0.0018)^(0.5)];

\draw (3.25,1)node[below]{$v_{4(1)(1)}$} -- (7,0);
\draw (3.25,1.3) ellipse[x radius=0.8,y radius=0.3];
\node at (3.25,1.3){$G_{v_{4(1)(1)}}$};
\filldraw[fill=black,draw=black] (3.25,1) circle [radius=(0.0018)^(0.5)];

\filldraw[fill=black,draw=black](3.8125,1) circle [radius=0.01];
\filldraw[fill=black,draw=black](4.1875,1) circle [radius=0.01];
\filldraw[fill=black,draw=black](4.5625,1) circle [radius=0.01];

\draw (5.5,1) -- (7,0);
\node at (5.05,0.755){$v_{4(1)(m_{3})}$};
\draw (5.5,1.3) ellipse[x radius=1,y radius=0.3];
\node at (5.5,1.3){$G_{v_{4(1)(m_{3})}}$};
\filldraw[fill=black,draw=black] (5.5,1) circle [radius=(0.0018)^(0.5)];

\draw (7.5,1)node[above]{$v_{4(1)(m_{3}+1)}$} -- (7,0);
\filldraw[fill=black,draw=black] (7.5,1) circle [radius=(0.0018)^(0.5)];

\filldraw[fill=black,draw=black](8.625,1) circle [radius=0.01];
\filldraw[fill=black,draw=black](9.375,1) circle [radius=0.01];
\filldraw[fill=black,draw=black](10.125,1) circle [radius=0.01];

\draw (13,1)node[above]{$v_{4(1)(\varDelta -2)}$} -- (7,0);
\filldraw[fill=black,draw=black] (13,1) circle [radius=(0.0018)^(0.5)];

\draw (10,-0.6) ellipse[x radius=0.3,y radius=0.6];
\node at (10,-0.6){$G_{v_{5}}$};

\draw (12,-0.6) ellipse[x radius=0.3,y radius=0.6];
\node at (12,-0.6){$G_{v_{q}}$};

\filldraw[fill=black,draw=black] (0,0) circle [radius=(0.0018)^(0.5)]; 
\filldraw[fill=black,draw=black] (2,0) circle [radius=(0.0018)^(0.5)]; 
\filldraw[fill=black,draw=black] (4,0) circle [radius=(0.0018)^(0.5)]; 
\filldraw[fill=black,draw=black] (7,0) circle [radius=(0.0018)^(0.5)]; 
\filldraw[fill=black,draw=black] (10,0) circle [radius=(0.0018)^(0.5)]; 
\filldraw[fill=black,draw=black] (12,0) circle [radius=(0.0018)^(0.5)]; 
\filldraw[fill=black,draw=black] (14,0) circle [radius=(0.0018)^(0.5)];

\node at (7,-2){$T$};
\end{tikzpicture}
\end{center}

\begin{center}
\begin{tikzpicture}

\draw (2,0)node[below]{$v_{2}$} -- (4,0)node[above]{$v_{3}$} -- (7,0)node[below]{$v_{4}$} -- (10,0)node[above]{$v_{5}$};

\filldraw[fill=black,draw=black](10.7,0) circle [radius=0.01];
\filldraw[fill=black,draw=black](11,0) circle [radius=0.01];
\filldraw[fill=black,draw=black](11.3,0) circle [radius=0.01];

\draw (12,0)node[above]{$v_{q}$} -- (14,0)node[below]{$v_{q+1}$};

\draw (1.25,-1)node[below]{$v_{3(1)(1)}$} -- (4,0);
\filldraw[fill=black,draw=black] (1.25,-1) circle [radius=(0.0018)^(0.5)];

\filldraw[fill=black,draw=black](1.75,-1) circle [radius=0.01];
\filldraw[fill=black,draw=black](2.05,-1) circle [radius=0.01];
\filldraw[fill=black,draw=black](2.35,-1) circle [radius=0.01];

\draw (3.25,-1)node[below]{$v_{3(1)(m_{2})}$} -- (4,0);
\filldraw[fill=black,draw=black] (3.25,-1) circle [radius=(0.0018)^(0.5)];

\filldraw[fill=black,draw=black](6.5,-1) circle [radius=0.01];
\filldraw[fill=black,draw=black](6.85,-1) circle [radius=0.01];
\filldraw[fill=black,draw=black](7.15,-1) circle [radius=0.01];

\draw (6,-1)node[below]{$v_{3(1)(m_{2}+1)}$} -- (4,0);
\filldraw[fill=black,draw=black] (6,-1) circle [radius=(0.0018)^(0.5)];

\draw (8,-1)node[below right]{$v_{3(1)(\varDelta -2)}$} -- (4,0);
\filldraw[fill=black,draw=black] (8,-1) circle [radius=(0.0018)^(0.5)];

\draw (3.25,1)node[below]{$v_{4(1)(1)}$} -- (7,0);
\draw (3.25,1.3) ellipse[x radius=0.8,y radius=0.3];
\node at (3.25,1.3){$G_{v_{4(1)(1)}}$};
\filldraw[fill=black,draw=black] (3.25,1) circle [radius=(0.0018)^(0.5)];

\filldraw[fill=black,draw=black](3.8125,1) circle [radius=0.01];
\filldraw[fill=black,draw=black](4.1875,1) circle [radius=0.01];
\filldraw[fill=black,draw=black](4.5625,1) circle [radius=0.01];

\draw (5.5,1) -- (7,0);
\node at (5.05,0.75){$v_{4(1)(m_{3})}$};
\draw (5.5,1.3) ellipse[x radius=1,y radius=0.3];
\node at (5.5,1.3){$G_{v_{4(1)(m_{3})}}$};
\filldraw[fill=black,draw=black] (5.5,1) circle [radius=(0.0018)^(0.5)];

\draw (7.9,1) -- (7,0);
\node at (6.75,0.8){$v_{4(1)(m_{3}+1)}$};
\draw (7.9,1.3) ellipse[x radius=0.8,y radius=0.3];
\node at (7.9,1.3){$G_{v_{2}}^{''}$};
\filldraw[fill=black,draw=black] (7.9,1) circle [radius=(0.0018)^(0.5)];

\filldraw[fill=black,draw=black](8.25,1) circle [radius=0.01];
\filldraw[fill=black,draw=black](8.5,1) circle [radius=0.01];
\filldraw[fill=black,draw=black](8.75,1) circle [radius=0.01];

\draw (12,1) -- (7,0);
\node at (10.1,1){$v_{4(1)(m_{3}+m_{2}+1)}$};
\draw (12,1.3) ellipse[x radius=1,y radius=0.3];
\node at (12,1.3){$G_{v_{3(1)(m_{2})}}$};
\filldraw[fill=black,draw=black] (12,1) circle [radius=(0.0018)^(0.5)];

\filldraw[fill=black,draw=black](13,1) circle [radius=0.01];
\filldraw[fill=black,draw=black](13.3,1) circle [radius=0.01];
\filldraw[fill=black,draw=black](13.6,1) circle [radius=0.01];

\draw (13.9,1)node[below]{$v_{4(1)(\varDelta -2)}$} -- (7,0);
\filldraw[fill=black,draw=black] (13.9,1) circle [radius=(0.0018)^(0.5)];

\draw (10,-0.6) ellipse[x radius=0.3,y radius=0.6];
\node at (10,-0.6){$G_{v_{5}}$};

\draw (12,-0.6) ellipse[x radius=0.3,y radius=0.6];
\node at (12,-0.6){$G_{v_{q}}$};

\filldraw[fill=black,draw=black] (2,0) circle [radius=(0.0018)^(0.5)]; 
\filldraw[fill=black,draw=black] (4,0) circle [radius=(0.0018)^(0.5)]; 
\filldraw[fill=black,draw=black] (7,0) circle [radius=(0.0018)^(0.5)]; 
\filldraw[fill=black,draw=black] (10,0) circle [radius=(0.0018)^(0.5)]; 
\filldraw[fill=black,draw=black] (12,0) circle [radius=(0.0018)^(0.5)]; 
\filldraw[fill=black,draw=black] (14,0) circle [radius=(0.0018)^(0.5)];

\node at (7,-2){$T^{*}$};

\node at (7,-3){$Figure\ 12: The\ graphs\ T\ and\ T^{*}\ in\ Lemma\ 15.$};
\end{tikzpicture}
\end{center}

\begin{proof}
Let $x$ be Perron eigenvector of $T^{*}$. By Lemma \ref{lem9}, we have
\begin{eqnarray}
& &\rho(T^{*})\cdot[x_{v_{4}}+3x_{v_{4(1)(m_{3}+1)}}-2x_{v_{2}}]>[7x_{v_{2}} + 7\sum\limits_{i=1}^{\varDelta -2}{x_{v_{3(1)(i)}}} - 3S(G_{v_{2}}^{''})]
+4x_{v_{2}}-x_{v_{4}}-5x_{v_{4(1)(m_{3}+1)}}\nonumber \\
&\Rightarrow&[\rho(T^{*})+ 2]\cdot[x_{v_{4}}+3x_{v_{4(1)(m_{3}+1)}}-2x_{v_{2}}] > x_{v_{4}}+x_{v_{4(1)(m_{3}+1)}} > 0\nonumber \\
&\Rightarrow&x_{v_{4}}+3x_{v_{4(1)(m_{3}+1)}}>2x_{v_{2}}.\label{15.1}
\end{eqnarray}
\indent Similarly we can prove that for $j\in [1, m_{2}]$
\begin{eqnarray}
x_{v_{4}} + 3x_{v_{4(1)(m_{3}+j+1)}} >2x_{v_{3(1)(j)}}.\label{15.2}
\end{eqnarray}
\indent Thus by (\ref{15.1}) and (\ref{15.2}), we have
\begin{eqnarray} 
& &\frac{1}{2}[\rho(T)-\rho(T^{*})] \nonumber \\
&\geq&\frac{1}{2}x^{t}[D(T)-D(T^{*})]x  \nonumber \\
&=&[x_{v_{1}}+\sum\limits_{i=1}^{m_{1}}{x_{v_{2(1)(i)}}}]\cdot \{3x_{v_{4(1)(m_{3}+1)}}+1\cdot[x_{v_{4}}+\sum\limits_{i=1}^{m_{3}}{x_{v_{4(1)(i)}}}+\sum\limits_{i=m_{3}+2}^{\varDelta -2}{x_{v_{4(1)(i)}}}+ \sum\limits_{i=1}^{m_{3}}{S(G_{v_{4(1)(i)}})}\nonumber \\
& &\ \ \ \ \ \ \ \ \ \ \ \ \ \ \ \ \ \ \ \ \ \ \ \ \ \ \ +\sum\limits_{i=5}^{q}{S(G_{v_{i}})}+\sum\limits_{i=5}^{q+1}{x_{v_{i}}}]-[x_{v_{3}}+\sum\limits_{i=1}^{\varDelta -2}{x_{v_{3(1)(i)}}}]-3x_{v_{2}}\} \nonumber \\
& &+ \sum\limits_{j=1}^{m_{2}}{S(G_{v_{3(1)(j)}})}\cdot \{3x_{v_{4(1)(m_{3}+j+1)}}+[x_{v_{4}}+\sum\limits_{i=1}^{m_{3}+j}{x_{v_{4(1)(i)}}}+\sum\limits_{i=m_{3}+j+2}^{\varDelta -2}{x_{v_{4(1)(i)}}}+\sum\limits_{i=1}^{m_{3}}{S(G_{v_{4(1)(i)}})}\nonumber \\
& &\ \ \ \ \ \ \ \ \ \ \ \ \ \ \ \ \ \ \ \ \ \ \ \ \ +\sum\limits_{i=5}^{q}{S(G_{v_{i}})}+\sum\limits_{i=5}^{q+1}{x_{v_{i}}}]-[x_{v_{2}}+x_{v_{3}}+\sum\limits_{i=1}^{j-1}{x_{v_{3(1)(i)}}}+ \sum\limits_{i=j+1}^{\varDelta -2}{x_{v_{3(1)(i)}}}]-3x_{v_{3(1)(j)}}\} \nonumber \\
&>&[x_{v_{1}}+\sum\limits_{i=1}^{m_{1}}{x_{v_{2(1)(i)}}}]\cdot \{[\sum\limits_{i=1}^{m_{3}}{x_{v_{4(1)(i)}}}+\sum\limits_{i=m_{3}+2}^{\varDelta -2}{x_{v_{4(1)(i)}}}+ \sum\limits_{i=1}^{m_{3}}{S(G_{v_{4(1)(i)}})}+\sum\limits_{i=5}^{q}{S(G_{v_{i}})}+\sum\limits_{i=5}^{q+1}{x_{v_{i}}}]\nonumber \\
& &\ \ \ \ \ \ \ \ \ \ \ \ \ \ \ \ \ \ \ \ \ \ \ \ \ \ \ -[x_{v_{3}}+\sum\limits_{i=1}^{\varDelta -2}{x_{v_{3(1)(i)}}}]-x_{v_{2}}\} \label{15.3} \\
& &+ \sum\limits_{j=1}^{m_{2}}{S(G_{v_{3(1)(j)}})}\cdot \{[\sum\limits_{i=1}^{m_{3}+j}{x_{v_{4(1)(i)}}}+\sum\limits_{i=m_{3}+j+2}^{\varDelta -2}{x_{v_{4(1)(i)}}}+\sum\limits_{i=1}^{m_{3}}{S(G_{v_{4(1)(i)}})}+\sum\limits_{i=5}^{q}{S(G_{v_{i}})}+\sum\limits_{i=5}^{q+1}{x_{v_{i}}}]\nonumber \\
& &\ \ \ \ \ \ \ \ \ \ \ \ \ \ \ \ \ \ \ \ \ \ \ \ \ -[x_{v_{2}}+x_{v_{3}}+\sum\limits_{i=1}^{j-1}{x_{v_{3(1)(i)}}}+ \sum\limits_{i=j+1}^{\varDelta -2}{x_{v_{3(1)(i)}}}]-x_{v_{3(1)(j)}}\}.\label{15.4}
\end{eqnarray}
\indent If $(\ref{15.3})< 0$, then we get
\begin{eqnarray}
\rho(T^{*})(x_{v_{3}}-x_{v_{5}})&=&2[\sum\limits_{i=5}^{q+1}{x_{v_{i}}}+\sum\limits_{i=5}^{q}{S(G_{v_{i}})}-\sum\limits_{i=2}^3{x_{v_{i}}}-\sum\limits_{i=1}^{\varDelta -2}{x_{v_{3(1)(i)}}}] <0\Rightarrow x_{v_{3}}<x_{v_{5}}, \label{15.5} \\
\rho(T^{*})(x_{v_{4}}-x_{v_{5}})&<&\sum\limits_{i=5}^{q+1}{x_{v_{i}}}+\sum\limits_{i=5}^{q}{S(G_{v_{i}})}-\sum\limits_{i=2}^3{x_{v_{i}}}-\sum\limits_{i=1}^{\varDelta -2}{x_{v_{3(1)(i)}}} < 0\Rightarrow x_{v_{4}}<x_{v_{5}}.\label{15.6}
\end{eqnarray}
\indent By (\ref{15.5}), (\ref{15.6}) and Lemma \ref{lem5}, we have
\begin{eqnarray}
x_{v_{3}} < x_{v_{5}} < x_{v_{6}} < \cdots < x_{v_{q}} < x_{v_{q+1}}.\label{15.7}
\end{eqnarray}
\indent By (\ref{15.7}) and Lemma \ref{lem1}, we have
\begin{eqnarray}
x_{v_{2}} < x_{v_{q+1}}.\label{15.8}
\end{eqnarray}
\indent Let $v\in G_{v_{4(1)(1)}}^{p}\cup \cdots \cup G_{v_{4(1)(m_{3})}}^{p}\cup G_{v_{5}}^{p}\cup \cdots \cup G_{v_{q}}^{p}$, then similar to above, we can prove for $i\in [1,\varDelta -2]$
\begin{equation}\label{15.9}
x_{v_{3(1)(i)}} < x_{v}.
\end{equation}
\indent Note that $\sum\limits_{i=1}^{m_{3}}{|G_{v_{4(1)(i)}}^{p}|}+\sum\limits_{i=5}^q{|G_{v_{i}}^{p}|}\geq |v_{3(1)}|=\varDelta -2$, and by (\ref{15.9}), we have
\begin{equation}\label{15.10}
\sum\limits_{i=1}^{\varDelta -2}{x_{v_{3(1)(i)}}} < \sum\limits_{i=1}^{m_{3}}{S(G_{v_{4(1)(i)}}^{p})}+\sum\limits_{i=5}^{q}{S(G_{v_{i}}^{p})}.
\end{equation}
\indent  By (\ref{15.7}), (\ref{15.8}) and (\ref{15.10}), we know $(\ref{15.3})> 0$, a contradiction, so $(\ref{15.3})\geq 0$. And similarly, we can prove $(\ref{15.4})\geq 0$. Thus we get the conclusion $\rho(T)>\rho(T^{*})$.
\end{proof}

\begin{theorem}\label{lem16} 
Let $P=v_{1}v_{2}\cdots v_{q}v_{q+1}$ be a path of tree $T$ ($q\geq 5$). Let $G_{v_{2}}^{''}=G_{v_{2}}\cup \{v_{1}\}\cup  v_{1}v_{4(1)(m_{3}+1)}$. $T$ satisfies following conditions:
\begin{enumerate}
\item $|G_{v_{2}}|=|v_{2(1)}|=m_{1}$, $|G_{v_{4(1)(m_{3})}}| = |v_{4(2)(m_{3})}| >0$, $|\{v_{1}\}|+ |G_{v_{2}}| + |G_{v_{4(1)(m_{3})}}|\leq \varDelta - 1$.
\item $\sum\limits_{i=1}^{m_{3}-1}{|G_{v_{4(1)(i)}}^{p}|}+\sum\limits_{i=5}^q{|G_{v_{i}}^{p}|}\geq |v_{3(1)}|=\varDelta -2$.
\item $|G_{v_{3(1)(i)}}| = |v_{3(2)(i)}| = \varDelta -2$, $i\in [1,m_{2}]$.
\item $deg(v_{3(1)(i)})=1$, $i\in [m_{2}+1,\varDelta-2]$.
\item $(m_{2}+1)+m_{3}= \varDelta -1$.
\end{enumerate}
Let $T^{*}=T-(\sum\limits_{i=1}^{m_{1}}{v_{2}v_{2(1)(i)}}+v_{2}v_{1}+\sum\limits_{i=1}^{m_{2}}{v_{3}v_{3(1)(i)}}) + (v_{4(1)(m_{3}+1)}v_{1}+\sum\limits_{i=1}^{m_{1}}{v_{4(1)(m_{3}+1)}v_{2(1)(i)}}+\sum\limits_{i=1}^{m_{2}}{v_{4(1)(m_{3}+i+1)}v_{3(1)(i)}})$, then $\rho(T)>\rho(T^{*})$.
\end{theorem}

\begin{center}
\begin{tikzpicture}
\draw (0,0)node[below]{$v_{1}$} -- (2,0)node[below]{$v_{2}$} -- (4,0)node[above]{$v_{3}$} -- (7,0)node[below]{$v_{4}$} -- (10,0)node[above]{$v_{5}$};

\filldraw[fill=black,draw=black](10.7,0) circle [radius=0.01];
\filldraw[fill=black,draw=black](11,0) circle [radius=0.01];
\filldraw[fill=black,draw=black](11.3,0) circle [radius=0.01];

\draw (12,0)node[above]{$v_{q}$} -- (14,0)node[below]{$v_{q+1}$};

\draw (2,0.6) ellipse[x radius=0.3,y radius=0.6];
\node at (2,0.6){$G_{v_{2}}$};

\draw (1.25,-1)node[above left]{$v_{3(1)(1)}$} -- (4,0);
\filldraw[fill=black,draw=black] (1.25,-1) circle [radius=(0.0018)^(0.5)]; 
\draw (1.25,-1.3) ellipse[x radius=0.8,y radius=0.3];
\node at (1.25,-1.3){$G_{v_{3(1)(1)}}$};

\filldraw[fill=black,draw=black](1.75,-1) circle [radius=0.01];
\filldraw[fill=black,draw=black](2.05,-1) circle [radius=0.01];
\filldraw[fill=black,draw=black](2.35,-1) circle [radius=0.01];

\draw (3.25,-1) -- (4,0);
\node at (4.5,-0.75){$v_{3(1)(m_{2})}$};
\filldraw[fill=black,draw=black] (3.25,-1) circle [radius=(0.0018)^(0.5)]; 
\draw (3.25,-1.3) ellipse[x radius=1,y radius=0.3];
\node at (3.25,-1.3){$G_{v_{3(1)(m_{2})}}$};

\filldraw[fill=black,draw=black](6.5,-1) circle [radius=0.01];
\filldraw[fill=black,draw=black](6.85,-1) circle [radius=0.01];
\filldraw[fill=black,draw=black](7.15,-1) circle [radius=0.01];

\draw (6,-1)node[below]{$v_{3(1)(m_{2}+1)}$} -- (4,0);
\filldraw[fill=black,draw=black] (6,-1) circle [radius=(0.0018)^(0.5)];

\draw (8,-1)node[below right]{$v_{3(1)(\varDelta -2)}$} -- (4,0);
\filldraw[fill=black,draw=black] (8,-1) circle [radius=(0.0018)^(0.5)];

\draw (3.25,1)node[below]{$v_{4(1)(1)}$} -- (7,0);
\draw (3.25,1.3) ellipse[x radius=0.8,y radius=0.3];
\node at (3.25,1.3){$G_{v_{4(1)(1)}}$};
\filldraw[fill=black,draw=black] (3.25,1) circle [radius=(0.0018)^(0.5)];

\filldraw[fill=black,draw=black](3.8125,1) circle [radius=0.01];
\filldraw[fill=black,draw=black](4.1875,1) circle [radius=0.01];
\filldraw[fill=black,draw=black](4.5625,1) circle [radius=0.01];

\draw (5.5,1) -- (7,0);
\node at (5.05,0.755){$v_{4(1)(m_{3})}$};
\draw (5.5,1.3) ellipse[x radius=1,y radius=0.3];
\node at (5.5,1.3){$G_{v_{4(1)(m_{3})}}$};
\filldraw[fill=black,draw=black] (5.5,1) circle [radius=(0.0018)^(0.5)];

\draw (7.5,1)node[above]{$v_{4(1)(m_{3}+1)}$} -- (7,0);
\filldraw[fill=black,draw=black] (7.5,1) circle [radius=(0.0018)^(0.5)];

\filldraw[fill=black,draw=black](8.625,1) circle [radius=0.01];
\filldraw[fill=black,draw=black](9.375,1) circle [radius=0.01];
\filldraw[fill=black,draw=black](10.125,1) circle [radius=0.01];

\draw (13,1)node[above]{$v_{4(1)(\varDelta -2)}$} -- (7,0);
\filldraw[fill=black,draw=black] (13,1) circle [radius=(0.0018)^(0.5)];

\draw (10,-0.6) ellipse[x radius=0.3,y radius=0.6];
\node at (10,-0.6){$G_{v_{5}}$};

\draw (12,-0.6) ellipse[x radius=0.3,y radius=0.6];
\node at (12,-0.6){$G_{v_{q}}$};

\filldraw[fill=black,draw=black] (0,0) circle [radius=(0.0018)^(0.5)]; 
\filldraw[fill=black,draw=black] (2,0) circle [radius=(0.0018)^(0.5)]; 
\filldraw[fill=black,draw=black] (4,0) circle [radius=(0.0018)^(0.5)]; 
\filldraw[fill=black,draw=black] (7,0) circle [radius=(0.0018)^(0.5)]; 
\filldraw[fill=black,draw=black] (10,0) circle [radius=(0.0018)^(0.5)]; 
\filldraw[fill=black,draw=black] (12,0) circle [radius=(0.0018)^(0.5)]; 
\filldraw[fill=black,draw=black] (14,0) circle [radius=(0.0018)^(0.5)];

\node at (7,-2){$T$};
\end{tikzpicture}
\end{center}

\begin{center}
\begin{tikzpicture}

\draw (2,0)node[below]{$v_{2}$} -- (4,0)node[above]{$v_{3}$} -- (7,0)node[below]{$v_{4}$} -- (10,0)node[above]{$v_{5}$};

\filldraw[fill=black,draw=black](10.7,0) circle [radius=0.01];
\filldraw[fill=black,draw=black](11,0) circle [radius=0.01];
\filldraw[fill=black,draw=black](11.3,0) circle [radius=0.01];

\draw (12,0)node[above]{$v_{q}$} -- (14,0)node[below]{$v_{q+1}$};

\draw (1.25,-1)node[below]{$v_{3(1)(1)}$} -- (4,0);
\filldraw[fill=black,draw=black] (1.25,-1) circle [radius=(0.0018)^(0.5)];

\filldraw[fill=black,draw=black](1.75,-1) circle [radius=0.01];
\filldraw[fill=black,draw=black](2.05,-1) circle [radius=0.01];
\filldraw[fill=black,draw=black](2.35,-1) circle [radius=0.01];

\draw (3.25,-1)node[below]{$v_{3(1)(m_{2})}$} -- (4,0);
\filldraw[fill=black,draw=black] (3.25,-1) circle [radius=(0.0018)^(0.5)];

\filldraw[fill=black,draw=black](6.5,-1) circle [radius=0.01];
\filldraw[fill=black,draw=black](6.85,-1) circle [radius=0.01];
\filldraw[fill=black,draw=black](7.15,-1) circle [radius=0.01];

\draw (6,-1)node[below]{$v_{3(1)(m_{2}+1)}$} -- (4,0);
\filldraw[fill=black,draw=black] (6,-1) circle [radius=(0.0018)^(0.5)];

\draw (8,-1)node[below right]{$v_{3(1)(\varDelta -2)}$} -- (4,0);
\filldraw[fill=black,draw=black] (8,-1) circle [radius=(0.0018)^(0.5)];

\draw (3.25,1)node[below]{$v_{4(1)(1)}$} -- (7,0);
\draw (3.25,1.3) ellipse[x radius=0.8,y radius=0.3];
\node at (3.25,1.3){$G_{v_{4(1)(1)}}$};
\filldraw[fill=black,draw=black] (3.25,1) circle [radius=(0.0018)^(0.5)];

\filldraw[fill=black,draw=black](3.8125,1) circle [radius=0.01];
\filldraw[fill=black,draw=black](4.1875,1) circle [radius=0.01];
\filldraw[fill=black,draw=black](4.5625,1) circle [radius=0.01];

\draw (5.5,1) -- (7,0);
\node at (5.05,0.755){$v_{4(1)(m_{3})}$};
\draw (5.5,1.8) ellipse[x radius=1,y radius=0.8];
\node at (5.5,1.9){$G_{v_{4(1)(m_{3})}}$};
\node at (5.5,1.4){$\cup G_{v_{2}}^{''}$};
\filldraw[fill=black,draw=black] (5.5,1) circle [radius=(0.0018)^(0.5)];

\draw (7.5,1)node[below right]{$v_{4(1)(m_{3}+1)}$} -- (7,0);
\draw (7.5,1.3) ellipse[x radius=0.8,y radius=0.3];
\node at (7.5,1.3){$G_{v_{3(1)(1)}}$};
\filldraw[fill=black,draw=black] (7.5,1) circle [radius=(0.0018)^(0.5)];

\filldraw[fill=black,draw=black](8.625,1) circle [radius=0.01];
\filldraw[fill=black,draw=black](9.375,1) circle [radius=0.01];
\filldraw[fill=black,draw=black](10.125,1) circle [radius=0.01];

\draw (13,1)node[below right]{$v_{4(1)(\varDelta -2)}$} -- (7,0);
\draw (13,1.3) ellipse[x radius=1,y radius=0.3];
\node at (13,1.3){$G_{v_{3(1)(m_{2})}}$};
\filldraw[fill=black,draw=black] (13,1) circle [radius=(0.0018)^(0.5)];

\draw (10,-0.6) ellipse[x radius=0.3,y radius=0.6];
\node at (10,-0.6){$G_{v_{5}}$};

\draw (12,-0.6) ellipse[x radius=0.3,y radius=0.6];
\node at (12,-0.6){$G_{v_{q}}$};

\filldraw[fill=black,draw=black] (2,0) circle [radius=(0.0018)^(0.5)]; 
\filldraw[fill=black,draw=black] (4,0) circle [radius=(0.0018)^(0.5)]; 
\filldraw[fill=black,draw=black] (7,0) circle [radius=(0.0018)^(0.5)]; 
\filldraw[fill=black,draw=black] (10,0) circle [radius=(0.0018)^(0.5)]; 
\filldraw[fill=black,draw=black] (12,0) circle [radius=(0.0018)^(0.5)]; 
\filldraw[fill=black,draw=black] (14,0) circle [radius=(0.0018)^(0.5)]; 

\node at (7,-2){$T^{*}$};

\node at (7,-3){$Figure\ 13: The\ graphs\ T\ and\ T^{*}\ in\ Lemma\ 16.$};
\end{tikzpicture}
\end{center}

\begin{proof}
\indent Similar as Lemma \ref{lem15}, we can prove $\rho(T)>\rho(T^{*})$.
\end{proof}

\section{Proof of Main Theorem}

\begin{theorem0}
Let $T_{(n,\varDelta)}$ be the tree with order n and maximum degree $\varDelta$, where $2\varDelta+1 \leq n \leq \varDelta^{2} +1$ and $\varDelta \geq 3$. $T\in T_{(n,\varDelta)}$ and $\rho(T)=\rho_{min}(T_{(n,\varDelta)})$. Then $T$ is the complete $\varDelta$-ary tree.
\end{theorem0}
\begin{proof}
{\bfseries Case 1:} $2\varDelta+1 \leq n \leq 3\varDelta-1$.\\

{\bfseries Case 1.1:} $diam(T)\geq 7$.\\
\indent Let $v_{1}v_{2}\cdots v_{8}\cdots v_{d+1}$ be diametrical path of $T$.As by Lemma \ref{lem8}, we have 
 $$3\varDelta-1 \geq n \geq \sum\limits_{i=2}^7{|v_{i(1)}|}+\sum\limits_{i=1}^8{|\{v_{i}\}|} \geq 3(\varDelta-2)+8 =3\varDelta+2,$$ 
which is a contradiction. So, $diam(T)\leq 6$.\\

{\bfseries Case 1.2:} $diam(T)=6$. \\
\indent Let $v_{1}v_{2}\cdot \cdot \cdot v_{7}$ be diametrical path of $T$. Without loss of generality, assume that $|G_{v_{6}}|\geq|G_{v_{2}}|$, i.e., $|v_{6(1)}| \geq |v_{2(1)}|$.\\

{\bfseries Claim 1}: $|v_{3(1)}|=\varDelta-2$.\\
\indent Otherwise, by Lemma \ref{lem8} and Lemma \ref{lem13}, we can get $T^{*}$ with $\rho(T^{*})<\rho(T)$ and $|v_{3(1)}|_{T^{*}}=\varDelta-2$, a contradiction.\\

{\bfseries Claim 2}: $|v_{5(1)}|=\varDelta-2$.\\
\indent Since $|G_{v_{3}}^{p}| \geq |v_{3(1)}| = \varDelta -2 \geq |G_{v_{6}}|,$ then using similar approach as of Claim 1 of Case 1.2, we can prove this Claim.\\

{\bfseries Claim 3}: $|v_{4(1)}|=\varDelta-2$.\\
\indent Otherwise, by Lemma \ref{lem7} or Lemma \ref{lem14}, we can get $T^{*}$ with $\rho(T^{*})<\rho(T)$ and $\varDelta-2\geq |v_{4(1)}|_{T^{*}}> |v_{4(1)}|_{T}$, a contradiction.\\

\indent Combining all the results of Case 1.2, we have $$3\varDelta-1 \geq n \geq \sum\limits_{i=3}^5{|v_{i(1)}|} + \sum\limits_{i=1}^7{|\{v_{i}\}|} \geq 3(\varDelta-2)+7 =3\varDelta+1,$$
a contradiction, so $diam(T) \leq 5$.\\

{\bfseries Case 1.3:} $diam(T)=5$. \\
\indent Let $v_{1}v_{2}\cdot \cdot \cdot v_{6}$ be diametrical path of $T$.\\

{\bfseries Claim 1}: $|v_{3(1)}|=|v_{4(1)}|=\varDelta-2$. \\
\indent Proof is similar as that of Claim 1 and Claim 2 of Case 1.2.\\

{\bfseries Claim 2}: $|v_{3(2)}|=0$. \\
\indent Otherwise, since $$|\{v_{1}\}|+|G_{v_{2}}|+|v_{3(2)}| \leq n-(|v_{3(1)}|+|v_{4(1)}|+\sum\limits_{i=2}^6{|\{v_{i}\}|}) \leq (3\varDelta-1)-2(\varDelta-2)-5 =\varDelta - 2,$$
by Lemma \ref{lem10}, we can get $T^{*}$ with $\rho(T^{*})<\rho(T)$ and $|v_{3(2)}|_{T^{*}}=0$, a contradiction. \\

{\bfseries Claim 3}: $|v_{4(2)}|=0$.\\ 
\indent Proof is similar as that of Claim 2 of Case 1.3.\\

{\bfseries Claim 4}: $|\{v_{1}\}|+|G_{v_{2}}|+|\{v_{6}\}|+|G_{v_{5}}|\leq \varDelta-1$. \\
$$|\{v_{1}\}|+|G_{v_{2}}|+|\{v_{6}\}|+|G_{v_{5}}| = n-(|G_{v_{3}}|+|G_{v_{4}}|+\sum\limits_{i=2}^5{|\{v_{i}\}|})\leq (3\varDelta-1)-2(\varDelta-2)-4 = \varDelta-1.$$\\ 
\indent Combining all the results of Case 1.3, by Lemma \ref{lem11}, we can get $T^{*}$ with $\rho(T^{*})<\rho(T)$ and $diam(T^{*})=4$, a contradiction. So, $diam(T)\leq4$.\\

{\bfseries Case 1.4:} $diam(T)=4$. \\
\indent Let $v_{1}v_{2}\cdot \cdot \cdot v_{5}$ be diametrical path of $T$.\\

{\bfseries Claim 1}: $|v_{3(1)}|=\varDelta-2$. \\
\indent Proof is similar as that of Claim 1 of Case 1.2.\\

\indent Assume there are $m^{'}$ $\frac{1}{2}$-pendant vertices attached to $v_{3}$ except for  $v_{2}$ and $v_{4}$. Without losing generality, assume $$|\{v_{5}\}|+|G_{v_{4}}|\geq |\{v_{1}\}|+|G_{v_{2}}|\geq |v_{3(2)(1)}|\geq \cdots \geq |v_{3(2)(m^{'})}|> |v_{3(2)(m^{'}+1)}| =\cdots =|v_{3(2)(\varDelta-2)}|=0.$$

{\bfseries Claim 2}: $m^{'}=0$ and $|G_{v_{4}}|=\varDelta-2$. \\
\indent Otherwise, $$|\{v_{1}\}|+|G_{v_{2}}|+|\{v_{5}\}|+|G_{v_{4}}|+|v_{3(2)}| = n-(|v_{3(1)}|+\sum\limits_{i=2}^4{|v_{i}|}) \leq (3\varDelta-1)-(\varDelta-2)-3 =2(\varDelta-1),$$
\noindent and note that $n \geq 2\varDelta +1$, then by Lemma \ref{lem10}, we can get $T^{*}$ with $\rho(T^{*})<\rho(T)$, $|G_{v_{4}}|_{T^{*}}=\varDelta-2$ and $m^{'}_{T^{*}}=0$, a contradiction.\\

{\bfseries Case 1.5:} $diam(T)<4$. \\
\indent Impossible since $n\geq 2\varDelta+1$.\\

\indent Combining all the results from Case 1.1 to Case 1.5, we can get the conclusion that $T$ is the complete $\varDelta$-ary tree under Case 1, i.e. $2\varDelta+1 \leq n \leq 3\varDelta-1$.\\

{\bfseries Case 2:} $3\varDelta \leq n \leq \varDelta^{2} +1$.\\
\begin{center}
\begin{tikzpicture}[scale=0.75]
\draw (0,0)node[below]{$v_{1}$} -- (2,0)node[below]{$v_{2}$} -- (4,0)node[below]{$v_{3}$} -- (6,0)node[below]{$v_{4}$} -- (8,0)node[below]{$v_{5}$} -- (10,0)node[below]{$v_{6}$} -- (12,0)node[below]{$v_{7}$};

\filldraw[fill=black,draw=black](12.9,0) circle [radius=0.01];
\filldraw[fill=black,draw=black](13,0) circle [radius=0.01];
\filldraw[fill=black,draw=black](13.1,0) circle [radius=0.01];

\draw (14,0)node[below]{$v_{d-1}$} -- (16,0)node[below]{$v_{d}$} -- (18,0)node[below]{$v_{d+1}$};

\draw (2,0.6) ellipse[x radius=0.5,y radius=0.6];
\node at (2,0.6){$G_{v_{2}}$};
\draw (4,0.6) ellipse[x radius=0.5,y radius=0.6];
\node at (4,0.6){$G_{v_{3}}$};
\draw (6,0.6) ellipse[x radius=0.5,y radius=0.6];
\node at (6,0.6){$G_{v_{4}}$};
\draw (8,0.6) ellipse[x radius=0.5,y radius=0.6];
\node at (8,0.6){$G_{v_{5}}$};
\draw (10,0.6) ellipse[x radius=0.5,y radius=0.6];
\node at (10,0.6){$G_{v_{6}}$};
\draw (12,0.6) ellipse[x radius=0.5,y radius=0.6];
\node at (12,0.6){$G_{v_{7}}$};
\draw (14,0.6) ellipse[x radius=0.7,y radius=0.6];
\node at (14,0.6){$G_{v_{d-1}}$};
\draw (16,0.6) ellipse[x radius=0.5,y radius=0.6];
\node at (16,0.6){$G_{v_{d}}$};

\filldraw[fill=black,draw=black] (0,0) circle [radius=(0.0018)^(0.5)]; 
\filldraw[fill=black,draw=black] (2,0) circle [radius=(0.0018)^(0.5)]; 
\filldraw[fill=black,draw=black] (4,0) circle [radius=(0.0018)^(0.5)]; 
\filldraw[fill=black,draw=black] (6,0) circle [radius=(0.0018)^(0.5)]; 
\filldraw[fill=black,draw=black] (8,0) circle [radius=(0.0018)^(0.5)]; 
\filldraw[fill=black,draw=black] (10,0) circle [radius=(0.0018)^(0.5)]; 
\filldraw[fill=black,draw=black] (12,0) circle [radius=(0.0018)^(0.5)]; 
\filldraw[fill=black,draw=black] (14,0) circle [radius=(0.0018)^(0.5)]; 
\filldraw[fill=black,draw=black] (16,0) circle [radius=(0.0018)^(0.5)]; 
\filldraw[fill=black,draw=black] (18,0) circle [radius=(0.0018)^(0.5)]; 

\node at (9,-1){$T^{'}$};

\node at (9,-1.5){$Figure\ 14: The\ graph\ T^{'}\ in\ Case\ 2.$};
\end{tikzpicture}
\end{center}

\begin{proof}
\indent It's obvious that by Lemma \ref{lem6}, any non-pre-tree with $3\varDelta \leq n \leq \varDelta^{2} +1$ can be transformed to a pre-tree $T^{'}$ with distance spectral radius decreased. Let $v_{1}v_{2}\cdots v_{7}\cdots v_{d+1}$ be an arbitary diametrical path of $T^{'}$. Without losing generality, we assume $|G_{v_{d}}|\geq|G_{v_{2}}|$. And it's obvious $|G_{v_{d-1}}^{p}| + |G_{v_{d}}| + |\{v_{d+1}\}|\geq \varDelta -1$. So by Lemma \ref{lem13}, Lemma \ref{lem7} (or Lemma \ref{lem14}), Lemma \ref{lem10} and Lemma \ref{lem15} (or Lemma \ref{lem16}), we can get $T^{'}_{1}$ with either $diam(T^{'}_{1}) = diam(T^{'})-1$ or $ndp(T^{'}_{1})\leq ndp(T^{'})-1$. By similar operation as above, we can decrease the diameter of this tree to 5.
\end{proof}

{\bfseries Case 2.1:} $diam(T)=5$.\\
\indent Let $v_{1}v_{2}\cdot \cdot \cdot v_{6}$ be diametrical path of $T$. \\

{\bfseries Claim 1:} $|v_{3(1)}|=|v_{4(1)}|=\varDelta -2$.\\
\indent Proof is similar as that of Claim 1 and Claim 2 of Case 1.2.\\

\indent Assume there are $m_{3}$ $\frac{1}{2}$-pendant vertices attached to $v_{3}$ except for $v_{2}$ and $v_{4}$. Assume there are $m_{4}$ $\frac{1}{2}$-pendant vertices attached to $v_{4}$ except for $v_{3}$ and $v_{5}$. Without losing generality, assume $m_{3}\geq m_{4}$, $$|v_{3(2)(m_{3})}|\geq \cdots \geq |v_{3(2)(2)}|\geq |v_{3(2)(1)}|\geq |\{v_{1}\}|+|G_{v_{2}}|$$ and $$|v_{4(2)(m_{4})}|\geq \cdots \geq |v_{4(2)(2)}|\geq |v_{4(2)(1)}|\geq |\{v_{6}\}|+|G_{v_{5}}|.$$\\

\begin{center}
\begin{tikzpicture}
\draw (0,0)node[below]{$v_{1}$} -- (2,0)node[above]{$v_{2}$} -- (4,0)node[below]{$v_{3}$} -- (7,0)node[above]{$v_{4}$} -- (9,0)node[below]{$v_{5}$} -- (11,0)node[above]{$v_{6}$};

\draw (2,-0.6) ellipse[x radius=0.3,y radius=0.6];
\node at (2,-0.6){$G_{v_{2}}$};

\draw (9,0.6) ellipse[x radius=0.3,y radius=0.6];
\node at (9,0.6){$G_{v_{5}}$};

\draw (1,1) -- (4,0);
\node at (0.7,0.7){$v_{3(1)(1)}$};
\filldraw[fill=black,draw=black](1.85,1)circle [radius=0.01];
\filldraw[fill=black,draw=black](2,1)circle [radius=0.01];
\filldraw[fill=black,draw=black](2.15,1)circle [radius=0.01]; 
\draw (4,1) -- (4,0);
\node at (3.3,0.8){$v_{3(1)(m_{3})}$};
\filldraw[fill=black,draw=black](4.85,1)circle [radius=0.01];
\filldraw[fill=black,draw=black](5,1)circle [radius=0.01];
\filldraw[fill=black,draw=black](5.15,1)circle [radius=0.01]; 
\draw (7,1) -- (4,0);
\node at (6.7,1.3){$v_{3(1)(\varDelta -2)}$};

\draw (1,1.5) ellipse[x radius=1.2,y radius=0.5];
\node at (1,1.5){$G_{v_{3(1)(1)}}$};

\draw (4,1.5) ellipse[x radius=1.2,y radius=0.5];
\node at (4,1.5){$G_{v_{3(1)(m_{3})}}$};

\draw (4,-1) -- (7,0);
\node at (3.7,-0.7){$v_{4(1)(1)}$};
\filldraw[fill=black,draw=black](4.85,-1)circle [radius=0.01];
\filldraw[fill=black,draw=black](5,-1)circle [radius=0.01];
\filldraw[fill=black,draw=black](5.15,-1)circle [radius=0.01]; 
\draw (7,-1) -- (7,0);
\node at (6.3,-0.7){$v_{4(1)(m_{4})}$};
\filldraw[fill=black,draw=black](7.85,-1)circle [radius=0.01];
\filldraw[fill=black,draw=black](8,-1)circle [radius=0.01];
\filldraw[fill=black,draw=black](8.15,-1)circle [radius=0.01]; 
\draw (10,-1) -- (7,0);
\node at (9.7,-1.3){$v_{4(1)(\varDelta -2)}$};

\draw (4,-1.5) ellipse[x radius=1.2,y radius=0.5];
\node at (4,-1.5){$G_{v_{4(1)(1)}}$};

\draw (7,-1.5) ellipse[x radius=1.2,y radius=0.5];
\node at (7,-1.5){$G_{v_{4(1)(m_{4})}}$};

\filldraw[fill=black,draw=black] (0,0) circle [radius=(0.0018)^(0.5)]; 
\filldraw[fill=black,draw=black] (2,0) circle [radius=(0.0018)^(0.5)]; 
\filldraw[fill=black,draw=black] (4,0) circle [radius=(0.0018)^(0.5)]; 
\filldraw[fill=black,draw=black] (7,0) circle [radius=(0.0018)^(0.5)]; 
\filldraw[fill=black,draw=black] (9,0) circle [radius=(0.0018)^(0.5)]; 
\filldraw[fill=black,draw=black] (11,0) circle [radius=(0.0018)^(0.5)];

\filldraw[fill=black,draw=black] (1,1) circle [radius=(0.0018)^(0.5)]; 
\filldraw[fill=black,draw=black] (4,1) circle [radius=(0.0018)^(0.5)]; 
\filldraw[fill=black,draw=black] (7,1) circle [radius=(0.0018)^(0.5)];

\filldraw[fill=black,draw=black] (4,-1) circle [radius=(0.0018)^(0.5)]; 
\filldraw[fill=black,draw=black] (7,-1) circle [radius=(0.0018)^(0.5)]; 
\filldraw[fill=black,draw=black] (10,-1) circle [radius=(0.0018)^(0.5)]; 

\node at (5.5,-2.2){$T$};

\node at (5.5,-2.7){$Figure\ 15: The\ graph\ T\ in\ Case\ 2.1.$};
\end{tikzpicture}
\end{center}

{\bfseries Claim 2}: $|v_{3(2)(i)}|=|v_{4(2)(j)}| =\varDelta-1$, $i\in [1, m_{3}]$, $j\in [1, m_{4}].$\\
\indent Otherwise, by Lemma \ref{lem10}, we can get $T^{*}$ with $\rho(T^{*})<\rho(T)$, $|v_{3(2)(i)}|_{T^{*}}=|v_{4(2)(j)}|_{T^{*}}=\varDelta-1$, $i\in [1, m_{3}^{'}]$, $j\in [1, m_{4}^{'}]$, $m_{3}^{'}\leq m_{3}$, $m_{4}^{'}\leq m_{4}$, a contradiction.\\

{\bfseries Case 2.1.1:} $m_{3}=m_{4}=0$.\\

{\bfseries Case 2.1.1.1:} $|\{v_{1}\}|+|G_{v_{2}}| + |\{v_{6}\}|+|G_{v_{5}}|\leq \varDelta-1$.\\
%\indent By Lemma \ref{lem11}, we can get $T^{*}$ with $\rho(T^{*})<\rho(T)$ and $diam(T^{*})=4$, a contradiction. So $diam(T)\leq 4$ under Case 2.1.1.1.\\
\indent Impossible since $n\geq 3\varDelta$.\\

{\bfseries Case 2.1.1.2:} $|\{v_{1}\}|+|G_{v_{2}}| + |\{v_{6}\}|+|G_{v_{5}}|> \varDelta-1$.\\

{\bfseries Case 2.1.1.2.1:} $|\{v_{1}\}|+|G_{v_{2}}| =\varDelta-1$ (or $|\{v_{6}\}|+|G_{v_{5}}|= \varDelta-1$).\\
\indent By Lemma \ref{lem15}, we can get $T^{*}$ with $\rho(T^{*})<\rho(T)$ and $diam(T^{*})=4$, a contradiction. So $diam(T)\leq 4$ under Case 2.1.1.2.1.\\

{\bfseries Case 2.1.1.2.2:} $|\{v_{1}\}|+|G_{v_{2}}| <\varDelta-1$ and $|\{v_{6}\}|+|G_{v_{5}}|< \varDelta-1$.\\
\indent By Lemma \ref{lem12}, we can get $T^{*}$ with $\rho(T^{*})<\rho(T)$ and $diam(T^{*})=4$, a contradiction. So $diam(T)\leq 4$ under Case 2.1.1.2.2.\\

{\bfseries Case 2.1.2:} $m_{3}\geq 1$ and $m_{4}\geq 0$.\\
\indent By Lemma \ref{lem15} or Lemma \ref{lem16}, we can get $T^{*}$ with $\rho(T^{*})<\rho(T)$ and $diam(T^{*})=4$, a contradiction. So $diam(T)\leq 4$ under Case 2.1.2.\\

\indent Combining results of Case 2.1.1 and Case 2.1.2, $diam(T)\leq 4$.\\

{\bfseries Case 2.2:} $diam(T)=4$.\\
\indent Let $v_{1}v_{2}\cdots v_{5}$ be diametrical path of $T$. \\

{\bfseries Claim 1:} $|v_{3(1)}|=\varDelta-2$.\\
\indent Proof is similar as that of Claim 1 of Case 1.2.\\

\indent Assume there are $m^{'}$ $\frac{1}{2}$-pendant vertices attached to $v_{3}$ except for $v_{2}$ and $v_{4}$. Without losing generality, assume $$
|\{v_{1}\}|+|G_{v_{2}}|\geq |\{v_{5}\}|+|G_{v_{4}}|\geq |v_{3(2)(1)}|\geq |v_{3(2)(2)}|\geq \cdots \geq |v_{3(2)(m^{'})}|.$$
\begin{center}
\begin{tikzpicture}

\draw (0,0)node[below]{$v_{1}$} -- (2,0)node[above]{$v_{2}$} -- (4,0)node[below]{$v_{3}$} -- (6,0)node[above]{$v_{4}$} -- (8,0)node[above]{$v_{5}$};

\draw (2,-0.6) ellipse[x radius=0.3,y radius=0.6];
\node at (2,-0.6){$G_{v_{2}}$};

\draw (6,-0.6) ellipse[x radius=0.3,y radius=0.6];
\node at (6,-0.6){$G_{v_{4}}$};

\draw (1,1) -- (4,0);
\node at (0.7,0.7){$v_{3(1)(1)}$};
\filldraw[fill=black,draw=black](1.85,1)circle [radius=0.01];
\filldraw[fill=black,draw=black](2,1)circle [radius=0.01];
\filldraw[fill=black,draw=black](2.15,1)circle [radius=0.01]; 
\draw (4,1) -- (4,0);
\node at (3.3,0.8){$v_{3(1)(m^{'})}$};
\filldraw[fill=black,draw=black](4.85,1)circle [radius=0.01];
\filldraw[fill=black,draw=black](5,1)circle [radius=0.01];
\filldraw[fill=black,draw=black](5.15,1)circle [radius=0.01]; 
\draw (7,1) -- (4,0);
\node at (6.7,1.3){$v_{3(1)(\varDelta -2)}$};

\draw (1,1.5) ellipse[x radius=1.2,y radius=0.5];
\node at (1,1.5){$G_{v_{3(1)(1)}}$};

\draw (4,1.5) ellipse[x radius=1.2,y radius=0.5];
\node at (4,1.5){$G_{v_{3(1)(m^{'})}}$};

\filldraw[fill=black,draw=black] (0,0) circle [radius=(0.0018)^(0.5)]; 
\filldraw[fill=black,draw=black] (2,0) circle [radius=(0.0018)^(0.5)]; 
\filldraw[fill=black,draw=black] (4,0) circle [radius=(0.0018)^(0.5)]; 
\filldraw[fill=black,draw=black] (6,0) circle [radius=(0.0018)^(0.5)]; 
\filldraw[fill=black,draw=black] (8,0) circle [radius=(0.0018)^(0.5)]; 
\filldraw[fill=black,draw=black] (8,0) circle [radius=(0.0018)^(0.5)];

\filldraw[fill=black,draw=black] (1,1) circle [radius=(0.0018)^(0.5)]; 
\filldraw[fill=black,draw=black] (4,1) circle [radius=(0.0018)^(0.5)]; 
\filldraw[fill=black,draw=black] (7,1) circle [radius=(0.0018)^(0.5)]; 

\node at (4,-1.2){$T$};

\node at (4,-2.2){$Figure\ 16: The\ graph\ T\ in\ Case\ 2.2.$};
\end{tikzpicture}
\end{center}

{\bfseries Claim 2:} $|\{v_{1}\}|+|G_{v_{2}}|= |\{v_{5}\}|+|G_{v_{4}}|=|v_{3(2)(i)}|=\varDelta-1 \geq |v_{3(2)(m^{'})}|$, $i\in [1,m^{'}-1]$.\\
\indent Otherwise, by Lemma \ref{lem10}, we can get $T^{*}$ with $\rho(T^{*})<\rho(T)$, $|\{v_{1}\}|+|G_{v_{2}}|= |\{v_{5}\}|+|G_{v_{4}}|=|v_{3(2)(i)}|=\varDelta-1 \geq |v_{3(2)(m^{'})}|$, $i\in [1,m^{'}-1]$, a contradiction.\\

{\bfseries Case 2.3:} $diam(T)<4$.\\
\indent Impossible since $n\geq 3\varDelta$.\\

\indent Combining all the results of Case 2, we get the conclusion that $T$ is the complete $\varDelta$-ary tree under Case 2, i.e. $3\varDelta \leq n \leq \varDelta^{2} +1$.\\
\end{proof}

\section{Acknowledgment} This research project is supported by the National Natural Science Foundation of China (RFIS) 12350410356. We thank the anonymous reviewer for the scrutiny and constructive comments, which helped improve the quality of this work.


\begin{thebibliography}{999}
\bibitem{ref1}
Stevanovi\'{c}, D.\& Ili\'{c}, A. (2010). Distance spectral radius of trees with fixed maximum degree. \textit{Electron. J. Linear Algebra}, 20, 168–179.

\bibitem{ref2}
Cvetkovi\'{c}, D.,  Doob, M.\& Sachs,H. (1995). \textit{Spectra of Graphs – Theory and Application}. (3rd ed). Johann Ambrosius Barth Verlag, Heidelberg–Leipzig.

\bibitem{ref3}
Graham, R. L. \& Pollack, H. O. (1971). On the addressing problem for loop switching.\textit{Bell Syst Tech J.}, 50, 2495–2519.

\bibitem{ref4}
Graham, R. L. \& Lovász, L. (1978). Distance matrix polynomials of trees. \textit{Adv Math.}, 29, 60–88.

\bibitem{ref5}
Edelberg, M., Garey, M. R. \& Graham, R. L. (1976). On the distance matrix of a tree.  \textit{Discrete Math.}, 14, 23–39.

\bibitem{ref6}
Consonni, V. \& Todeschini, R. (2008). New spectral indices for molecule description. \textit{MATCH Commun. Math. Comput. Chem.}, 60, 3–14.

\bibitem{ref7}
Gutman, I. \& Medeleanu, M. (1998). On the structure–dependence of the largest eigenvalue of the distance matrix of an alkane. \textit{Indian J. Chem. A}, 37, 569–573.

\bibitem{ref8}
Indulal, G. (2009). Sharp bounds on the distance spectral radius and the distance energy of graphs. \textit{ Linear Algebra Appl.}, 430, 106–113.

\bibitem{ref9}
Zhou, B. (2007). On the largest eigenvalue of the distance matrix of a tree. \textit{MATCH Commun. Math. Comput. Chem.}, 58, 657–662.

\bibitem{ref10}
Zhou, B. \& Ili\'{c}, A. (2010). On distance spectral radius and distance energy of graphs. \textit{MATCH Commun. Math. Comput. Chem.}, 64, 261–280.

\bibitem{ref11}
Zhou, B. \& Trinajsti\'{c}, N. (2007). On the largest eigenvalue of the distance matrix of a connected graph. \textit{Chem. Phys. Lett.}, 447, 384–387.

\bibitem{ref12}
Khan, Z. U. \& Zhang, X. D. (2023). On generalized distance spectral radius and generalized distance energy of graphs. \textit{Discrete Mathematics, Algorithms and Applications}, 15(08), 2250176.

\bibitem{ref13}
Bose, S. S., Nath, M. \& Paul, S. (2011). Distance spectral radius of graphs with r pendent vertices. \textit{Linear Algebra and its Applications}, 435(11),  2828 - 2836.

\bibitem{ref14}
Zhang, X. L. \& Godsil, C. (2011). Connectivity and minimal distance spectral radius of graphs. \textit{Linear and Multilinear Algebra}, 59(7), 745-754.

\bibitem{ref15}
Khan, Z. U. \& Zhang, X. D. (2022). On $ D_ {\alpha} $ spectrum of connected graphs. \textit{Contributions to Discrete Mathematics}, 17(2), 77-95.

\bibitem{ref16}
Yu, G. L., Jia, H. C., Zhang, H. L. \& Shu, J. L. (2012). Some graft transformations and its applications on the distance spectral radius of a graph. \textit{Appl. Math. Lett.}, 25(3), 315-319.

\bibitem{ref17}
Nath, M. \& Paul, S. (2012). On the distance spectral radius of bipartite graphs. \textit{Linear algebra and its applications}, 436(5), 1285 - 1296.

\bibitem{ref18}
Zhang, X.L. (2014). On the distance spectral radius of unicyclic graphs with perfect matchings. \textit{ Electronic Journal of Linear Algebra}, 27, 569- 587.

\bibitem{ref19}
Liu, H., Lu, M. \& Tian, F. (2007). On the spectral radius of unicyclic graphs with fixed diameter. \textit{Linear Algebra Appl.}, 420, 449-457.

\bibitem{ref20}
Subhi, R. \& Powers, D. (1990). The distance spectrum of the path $P_{n}$ and the first distance eigenvector of connected graphs. \textit{Linear Multilinear Algebra}, 28, 75–81.

\bibitem{ref21}
Aouchiche, M. \& Hansen, P. (2014). Distance spectra of graphs: a survey. \textit{Linear Algebra Appl}, 458, 301–386.

\bibitem{ref22}
Wei, N. (2016). \textit{On the Minimal Distance Spectral Radius of Tree with Given Maximal Degree}, Master Degree Thesis, China University of Petroleum.

\bibitem{ref23}
Stevanovi\'{c}, D. \& Ili\'{c}, A. Spectral properties of distance matrix of graphs, in: I. Gutman, B. Furtula (Eds.), Distance in Molecular Graphs Theory, in: Math. Chem. Monogr., vol. 12. \textit{University of Kragujevac, Kragujevac,} 2010, 139–176.

\bibitem{ref24}
Wang, Y. \& Zhou, B. (2013). On distance spectral radius of graphs. \textit{Linear Algebra Appl}, 438, 3490–3503.

\bibitem{ref25}
Du, Z., Ili\'{c}, A. \& Feng, L. (2013). Further results on distance spectral radius of graphs.\textit{Linear Multilinear Algebra}, 61, 1287–1301.

\end{thebibliography}
\end{document}